\documentclass[12pt, reqno]{amsart}
\usepackage{amsmath}
\usepackage{amsfonts}
\usepackage{amssymb}
\usepackage{comment}
\usepackage{graphicx}
\usepackage{xcolor}
\usepackage{orcidlink}
\usepackage[margin=1in,footskip=0.25in]{geometry}
\providecommand{\U}[1]{\protect\rule{.1in}{.1in}}
\newtheorem{theorem}{Theorem}

\newtheorem{corollary}[theorem]{Corollary}

\newtheorem{definition}[theorem]{Definition}
\newtheorem{example}[theorem]{Example}

\newtheorem{lemma}[theorem]{Lemma}

\newtheorem{remark}[theorem]{Remark}

\begin{document}

\title{Phase retrieval for nilpotent groups}

\author{Hartmut F\"{u}hr \orcidlink{0000-0001-5718-6449}, Vignon Oussa \orcidlink{0000-0003-2686-6470}}

\address{ H.~F{\"u}hr\\
Lehrstuhl f{\"u}r Geometrie und Analysis \\
RWTH Aachen University \\
 D-52056 Aachen\\
 Germany \\ email: fuehr@mathga.rwth-aachen.de}
 
\address{V. Oussa\\ Department of Mathematics \\ Bridgewater State University \\ MA 02234 \\ USA \\ email: vignon.oussa@bridgew.edu}
\begin{abstract}
We study the phase retrieval property for orbits of general irreducible representations of nilpotent groups, for the classes of simply connected connected Lie groups, and for finite groups. We prove by induction that in the Lie group case, all irreducible representations do phase retrieval. 

For the finite group case, we mostly focus on $p$-groups. Here our main result states that every irreducible representation of an arbitrary $p$-group with exponent $p$ and size $\le p^{2+p/2}$ does phase retrieval. 

Despite the fundamental differences between the two settings, our inductive proof methods are remarkably similar. 
\end{abstract}
\maketitle
\global\long\def\with{\,\middle|\,}

\hyphenation{an-iso-tropic}
\noindent \textbf{\small Keywords:}{\small {}  phase retrieval; nilpotent Lie group; finite nilpotent group; $p$-group; group frame;}{\small \par}

\noindent \textbf{\small AMS Subject Classification:}{\small {} 42C15;
42A38; 65T50; 94A12}{\small \par}

\section{Introduction}
Phase retrieval originated in the Fourier-analytic treatment of questions arising in crystallography. The phase retrieval problem for (finite) frames as a problem in mathematical signal processing was introduced in \cite{MR2224902}, with motivation stemming from applications such as speech recognition. Since its introduction phase retrieval has developed into a legitimate branch of applied mathematics, with links to diverse subjects such as harmonic analysis, algebraic geometry, optimization, numerical analysis, etc. We refer to the review papers \cite{MR4094471, MR4189292} and their extensive lists of references.

Despite the general interest in the problem, sources containing explicit constructions of frames with guaranteed phase retrieval properties are still relatively scarce. 
In this paper, we study the construction of such systems using strongly continuous, unitary representations of locally compact groups. Precisely, let $\pi$ be a continuous unitary representation of a locally compact group $G$ acting in some Hilbert space $\mathcal{H}_\pi.$ Fixing a vector $\eta \in \mathcal{H}_\pi,$ we define a continuous linear map $V_\eta : \mathcal{H}_\pi \to C(G)$ as $V_{\eta}(\psi)(g)=\langle \psi, \pi(g) \eta \rangle.$ Next, let  $\mathbb{T}$ denote the group of complex numbers of modulus one, and let $\mathcal{H}_\pi/\mathbb{T}$ denote the orbit space of $\mathcal{H}$ modulo the canonical scalar action of $\mathbb{T}$. We say that $\eta$ \textbf{does phase retrieval} if the map
\[
 \mathcal{A}_\eta : \mathcal{H}_\pi/\mathbb{T} \to C(G)~,~\mathcal{A}_\eta(\mathbb{T} \cdot f) (x) = |V_\eta f(x)|~,x\in G
\] is injective. Similarly, we say that \textbf{$\pi$ does phase retrieval} if there exists a vector $\eta \in \mathcal{H}_\pi$ that does phase retrieval.
It is worth noting that if $\eta \in \mathcal{H}_\pi$ does phase retrieval for $\pi$ then the linear map $V_{\eta}$ is necessarily injective. When this is the case, $V_\eta$ defines a linear bijection onto a subspace $\mathcal{C}_\eta  \subset C(G)$, and the phase retrieval problem is equivalent to the statement that each $F \in \mathcal{C}_\eta$ can be recovered from $|F|$ up to a scalar. This observation serves as motivation for the name `phase retrieval'; the phase of $F \in \mathcal{C}_\eta$ can be recovered from its modulus.  

Even though the phase retrieval problem for finite frames has been studied for over a decade, the bulk of the literature addressing this problem for group frames has only been devoted to some specific classes of groups, such as the finite Heisenberg groups  \cite{MR3500231,MR4222533,MR4094471}. The literature also contains a treatment of the affine Lie group (also known as the ax+b group) over the reals \cite{MR3421917,alaifari2020phase} as well as the finite ax+b group defined over prime fields \cite{bartuselphase}. Recently, the results of \cite{MR3500231} were extended to the case of projective unitary representations of finite abelian groups \cite{MR3901918}. To the best of our knowledge, this is the largest class of general group representations for which the phase retrieval property has been studied. The present work aims to significantly expand the class of available representations. Precisely, we show that the class of groups admitting representations for which phase retrieval is possible includes all simply connected, connected nilpotent Lie groups, as well as all sufficiently small $p$-groups with exponent $p$, for an arbitrary prime $p$. 

The following examples describe the group representations that are probably understood best so far, namely the Schr\"odinger representations of (finite or Lie) Heisenberg groups.  These will turn out to be the cornerstones of our general inductive approach. Throughout the paper, we let $\mathbb{Z}_n = \mathbb{Z}/n \mathbb{Z}$ be the finite cyclic group of order $n$, and denote elements of $\mathbb{Z}_n$ (by slight abuse of notation) by $k=0,\ldots,n-1$. \\

%\textcolor{red}{Do we want to relabel Remark 1 and Remark 2 as Example 1 and Example 2?} \\

%Let $\mathbb{Z}_n = \mathbb{Z}/n \mathbb{Z}$ be the finite cyclic group of order $n.$

\begin{example} \label{ex:finite_Heisenberg} (The finite Heisenberg group) Fix $n \in \mathbb{N}$, $n \ge 2$, and define the associated finite Heisenberg group by
 \begin{equation} \label{eqn:fH_set}
  \mathbb{H}_n = \mathbb{Z}_n^3
 \end{equation} with group law
 \begin{equation} \label{eqn:fH_law}
  (k,l,m)\cdot(k',l',m') = (k+k',l+l',m+m'-l'k)~,
 \end{equation} and arithmetic operations taken modulo $n$. The center of this group is given by $Z(\mathbb{H}) = \{ 0 \} \times \{ 0 \} \times \mathbb{Z}_n$, and $\mathbb{H}/Z(\mathbb{H}) \cong \mathbb{Z}_n \times \mathbb{Z}_n.$ Next, define the \textbf{Schr\"odinger representation} of $\mathbb{H}_n$ acting on $\mathbb{C}^n$ by first introducing the translation and modulation operators $T_k, M_l$ respectively as
 \[
  (T_k f)(y) = f(y-k)~,~(M_l f)(y) = e^{2 \pi i ly/n} f(y) ~.
 \]
Let $ \pi(k,l,m) = e^{2 \pi i m/n} M_l T_k~.$ Explicitly, 
 \[
  (\pi(k,l,m)f)(y) = e^{2 \pi im/n} e^{2 \pi i ly/n} f(y-k)~,k,l,m,y \in \{0,\ldots,n-1 \}~.
 \]  
This representation is known to be irreducible and also does phase retrieval \cite{MR3500231}. To construct a vector $\eta$ that does phase retrieval for the representation $\pi,$ define the \textit{ambiguity function} of $\eta,$ as follows:
\[
 A_\eta(k,l) = \langle \eta, M_l T_k \eta \rangle~.
\]
It is shown in \cite{MR3500231,MR3901918} that any vector $\eta$  for which the zero set of the matrix coefficient $A \eta$ is empty does phase retrieval. %In fact, a stronger property holds for these vectors, namely that $\eta$ has the \textit{maximal span property}. This means that 
%\[
%{\rm span} \{ (\pi(k,l,0) \eta) \otimes (\pi(k,l,0) \eta) : k,l = 0,\ldots,n-1 \} = \mathbb{C}^{n \times n}~.
%\] Here, we use the notation $z \otimes w$ for rank-one operators 
%\[
%z \otimes w : \mathbb{C}^n \ni v \mapsto \langle v,w \rangle z \in \mathbb{C}^n~~.
%\] 
%To see that the maximal span property implies phase retrieval, we first introduce the canonical scalar product $\langle A,B \rangle \mapsto {\rm trace}(AB^*)$ on $\mathbb{C}^{n \times n}$. Then the maximal span property implies that the linear map 
%\[ A \mapsto \left( \langle A, (\pi(k,l,0) \eta) \otimes (\pi(k,l,0) \eta) \rangle \right)_{k,l=0,\ldots,n-1} \]
%is one-to-one, hence invertible. For $A = f \otimes f$, one finds
%\[
% \langle A, (\pi(k,l,0) \eta) \otimes (\pi(k,l,0) \eta) \rangle = |\langle f, \pi(k,l,0) \eta \rangle|^2 = |V_\eta f(k,l,0)|^2~,
%\] which shows that $f \otimes f$ can be recovered from $|V_\eta f|$, if $\eta$ has the maximal span property. Since $f$ can be recovered from $f \otimes f$ up to a scalar multiple, this establishes phase retrieval for vectors $\eta$ with the maximal span property. 

Note that the existence of vectors with nonvanishing ambiguity function is established in Proposition 2.1 of \cite{MR3500231}.
\end{example}

\begin{example} \label{ex:Heisenberg_Lie} (The Heisenberg Lie group) Define the three-dimensional Heisenberg Lie group $\mathbb{H}$ as 
$\mathbb{H} = \mathbb{R}^3$, with group law 
\begin{equation}
(p,q,r) (p',q',r') = (p+p',q+q',r+r'+pq')~. 
\end{equation}
$\mathbb{H}$ is a simply connected, connected Lie group with center $\{ 0 \}\times \{ 0 \} \times \mathbb{R} \subset \mathbb{H}$. 
If we define the continuous-domain translation and modulation operators as $(T_p f)(y) = f(y-p)$, $(M_q f)(y) = e^{2 \pi i q y} f(y)$, acting on $f \in L^2(\mathbb{R})$, then $ \pi(p,q,r) = e^{2 \pi i r }  M_q T_p$ defines an irreducible, unitary representation $\pi$ of $\mathbb{H}$ on $L^2(\mathbb{R}).$ The representation $\pi$ is often referred to as the \textbf{Schr\"odinger representation}, and as its finite counterpart, it does phase retrieval, with an analogous sufficient criterion for vectors guaranteeing that property: Defining the ambiguity function of $\eta \in L^2(\mathbb{R})$ as 
\[
A_\eta (p,q) = \langle \eta, \pi(p,q,0) \eta \rangle~,
\] we have that $\eta$ does phase retrieval if $A_\eta$ vanishes only on a set of Lebesgue measure zero. A prominent example of a window $g$ with that property is provided by the Gaussian. 
\end{example}

\subsection{Main contributions of our paper}

In this paper, we intend to address the problem of phase retrieval for general irreducible representations of nilpotent groups. Our main results in this direction are the following. 
\begin{itemize}
    \item Any irreducible representation of an arbitrary simply connected, connected nilpotent Lie group does phase retrieval (Theorem \ref{thm:main_1}).
       \item Explicitly checkable sufficient criteria for the phase retrieval property of irreducible representations of $p$-groups (Corollary \ref{cor:p_groups_general}).
    \item As a corollary to the previous result: Any irreducible representation of an arbitrary $p$-group with exponent $p$ and size $\le p^{2+p/2}$ does phase retrieval (Corollary \ref{cor:small_p_groups}).
\end{itemize}

The techniques we use to establish these results seem of potentially independent interest. In particular, the treatment of the $p$-group case is surprisingly analogous to the Lie group case, due to a $p$-group version of Kirillov's lemma that appears to be new (Lemma \ref{Kirillov}), and an associated explicit construction method for irreducible representations (Lemma \ref{lem:exists_K_triple}). Under suitable assumptions on the underlying groups, this approach establishes a close relationship between irreducible representations of $p$-groups and finite Schr\"odinger representations that we systematically exploit for the study of phase retrieval. 

As a second technical novelty of potential independent interest, we introduce the quantity $p_0(\pi)$ (see Definition \ref{defn:p0}), which is an upper bound on the proportion of zeros in a nonzero matrix coefficient associated to $\pi$. Our results estimating $p_0(\pi)$, specifically Lemma \ref{lem:est_p0}, serve as further illustration of the usefulness of the $p$-group analog of the Kirillov lemma and its representation-theoretic ramifications. 

\section{Nilpotent Lie groups}

This section establishes phase retrieval for irreducible representations of simply connected, connected nilpotent Lie groups. Before we turn to groups and representations, we first state a few basic observations and ideas regarding phase retrieval, which will be repeatedly useful. We expect them to be well-known, but include short arguments for lack of a handy reference. 

\begin{definition}
Let $\mathcal{H}$ denote a Hilbert space. We say that $(x_i)_{i \in I} \subset \mathcal{H}$ \textbf{does phase retrieval} if 
\[
\forall f,g \in \mathcal{H}~:~ \left( \forall i \in I: |\langle f, x_i \rangle| =  |\langle g, x_i \rangle| \right) \Leftrightarrow f \in \mathbb{T} \cdot g ~.
\]
\end{definition}

With this definition, a unitary representation $\pi$ of a locally compact group does phase retrieval iff there exists a vector $\eta \in \mathcal{H}_\pi$ such that $(\pi(x) \eta)_{x \in G}$ does phase retrieval. Part (b) of the following lemma has been observed e.g. in \cite{MR4091190}. 

\begin{lemma} \label{lem:basics_pr}
Let $\mathcal{H}$ denote a Hilbert space and $(x_i)_{i \in I} \subset \mathcal{H}$.
\begin{enumerate}
    \item[(a)] If $(x_i)_{i \in I}$ does phase retrieval, then it has dense span in $\mathcal{H}$.
    \item[(b)] If $(x_i)_{i \in I}$ does phase retrieval, and $f,g \in \mathcal{H} \setminus \{ 0 \}$ are arbitrary, there exists $i \in I$ such that 
    \[
    \langle f, x_i \rangle \langle g, x_i \rangle \not= 0~. 
    \]
    \item[(c)] Let $(y_j)_{j \in J}$ be related to $(x_i)_{i \in I}$ by the following property: There exists a map $\kappa: J \to I$ such that for all $j \in J$
    \[
    y_j \in \mathbb{T} x_{\kappa(j)}
    \] and in addition, 
    \[
    \forall i \in I\exists j \in J~:~ x_i \in \mathbb{T} x_{\kappa(j)}~.
    \] Then $(x_i)_{i \in I}$ does phase retrieval iff $(y_j)_{j \in J}$ does. 
\end{enumerate}
\end{lemma}
\begin{proof}
Part (a) follows from the fact that any $f$ orthogonal to all $x_i$ must be in $\mathbb{T} \cdot 0 = \{ 0 \}$.

For part (b), assuming the existence of $f,g \in  \mathcal{H} \setminus \{ 0 \}$ with 
\[
 \forall i \in I~:~\langle f, x_i \rangle \langle g, x_i \rangle = 0~,
\] we first note that by part (a), the fact that $f,g$ are both nonzero implies that the associated coefficient families $(\langle f, x_i \rangle)_{i \in I}$ and $(\langle g,x_i \rangle)_{i \in I}$ are. By assumption, these coefficient families are disjointly supported. Hence they are linearly independent, and accordingly, $f$ and $g$ are linearly independent. In particular, $f + g \not\in \mathbb{T} (f-g)$. On the other hand, the disjointness assumption entails for all $i \in I$
\[
\left| \langle f+g, x_i \rangle \right| = \left| \langle f, x_i \rangle \right| + |\left \langle g, x_i \rangle \right| =  \left| \langle f-g, x_i \rangle \right|~,
\] which contradicts the phase retrieval property. 

For part (c) we note that the assumptions on the two families allow to recover the coefficient family $(\langle f, x_i \rangle)_{i \in I}$ from the coefficient family $(\langle f, y_j \rangle)_{j \in J}$, and the same holds for the families of coefficient moduli. Also, these observations remain true if the families $(x_i)_{i \in I}$ and $(y_j)_{j \in J}$ are interchanged. These facts then imply the equivalence in (c).  
\end{proof}

We next turn to Hilbert space tensor products and their realization via Hilbert-Schmidt operators. 
First some notation. Let $\mathcal{H}_1,\mathcal{H}_2$ denote nontrivial Hilbert spaces, and $HS(\mathcal{H}_2,\mathcal{H}_1)$ the space of Hilbert-Schmidt operators $A : \mathcal{H}_2 \to \mathcal{H}_1$, endowed with the usual scalar product 
\[
\langle A, B \rangle = {\rm trace}(AB^*)~.
\]
Given $x \in \mathcal{H}_1,y \in \mathcal{H}_2$, let 
\[
x \otimes y \in HS(\mathcal{H}_2,\mathcal{H}_1)~:~ \mathcal{H}_2 \ni z \mapsto \langle z,y \rangle x~. 
\] 

With this terminology in place, we can formulate the following basic observation. 
\begin{lemma} \label{lem:pr_HS}
Let $(x_i)_{i \in I} \subset \mathcal{H}_1$, $(y_j)_{j \in J} \subset \mathcal{H}_2$ denote two systems of vectors with the phase retrieval property. Then $(x_i \otimes y_j)_{(i,j) \in I \times J}$ has the phase retrieval property for $HS(\mathcal{H}_2,\mathcal{H}_1)$.
\end{lemma}
\begin{proof}
Let $A,B \in HS(\mathcal{H}_2,\mathcal{H}_1)$ with 
\[
\forall (i,j) \in I \times j~:~|\langle A, x_i \otimes y_j \rangle| = |\langle B, x_i \otimes y_j \rangle|
\]
Note that the operator scalar products are computed as 
\[
\langle A, x_i \otimes y_j \rangle = \langle A y_j, x_i \rangle~.
\]
Fixing $j \in J$, this gives for all $i \in I$:
\[
|\langle A y_j , x_i \rangle| = |\langle B y_j, x_i \rangle|~. 
\]
The phase retrieval property for $(x_i)_{i \in I}$ then entails 
\[
A y_j = v(j) B y_j~,~ \forall j \in J~,
\] with suitable $v(j) \in \mathbb{T}$. We next prove that $v$ is constant on the set
\[
S = \{ j \in J~:~ A y_j \not= 0 \}~. 
\]
To see this, we fix $i \in I$, and get for all $j \in J$ that
\[
|\langle y_j, A^* x_i \rangle| = |\langle y_j, B^* x_i \rangle|~.
\] Hence the phase retrieval property of $(y_j)_{j \in J}$ furnishes a map $u : I \to \mathbb{T}$ with 
\begin{equation} \label{eqn:A_phase_factor}
A^* x_i = u(i) B^* x_i~,~ \forall i \in I~.
\end{equation} Now assume that $j,j' \in S$ are given. Then there exists $i \in I$ such that 
\begin{equation} \label{eqn:int_supps_tensor}
\langle A y_j, x_i \rangle \langle A y_{j'},x_i \rangle \not= 0 ~,
\end{equation} by Lemma \ref{lem:basics_pr} (b), and definition of $S$. This entails 
\[
\langle A y_j, x_i \rangle = v(j) \langle B y_j, x_i \rangle =  v(j) u(i) \langle A y_j, x_i \rangle ~,
\] and the choice of $i$ allows to cancel $\langle A y_j, x_i \rangle$ from this equation to obtain $v(j) = \overline{u(i)}$. The same reasoning can be applied to $v(j')$, leading to $v(j') = \overline{u(i)} = v(j)$. Plugging this observation into equation (\ref{eqn:A_phase_factor}) provides $v \in \mathbb{T}$ satisfying
\[
\forall j \in J~:~  A y_j = vB y_j
\] and Lemma \ref{lem:basics_pr} (a) now applies to yield $A = vB$.
\end{proof}

We now turn to groups and representations. 
Before we state the main result of this section, we highlight several representation-theoretic facts relevant to the problem at hand (see Remark \ref{rem:pr_factor_group} and Lemma \ref{lem:ind_char}.) Given a locally compact group $G$, we let $\widehat{G}$ denote its unitary dual, i.e. the set of equivalence classes of irreducible unitary representations, and $\pi \in \widehat{G}$ will be shorthand for ``$\pi$ is irreducible''. Throughout the paper, we will repeatedly factor out suitably chosen normal subgroups. One such subgroup of particular interest is the \textit{projective kernel} of a unitary representation $\pi$, which is given by
\[
K_\pi = \{ x \in G : \pi(x) \in \mathbb{T} \cdot {\rm id}_{\mathcal{H}_\pi} \}~. 
\] This is a closed normal subgroup of $G$ containing the kernel of $\pi$.

\begin{remark} \label{rem:pr_factor_group}
 Let $G$ denote a locally compact group, and let $\pi$ be a unitary representation of $G$. Assume that $K \lhd G$ is a normal subgroup contained in the projective kernel of $\pi$. Then $\pi$ induces a projective representation $\overline{\pi}$ of the quotient group $G/K$, by picking a cross-section $\sigma : G/K \to G$ and letting $\overline{\pi}(xK) = \pi(\sigma(xK))$. Now Lemma \ref{lem:basics_pr}(c) shows that the system $(\pi(x) \eta)_{x \in G}$ does phase retrieval iff the system $(\pi(\sigma(x))_{xK \in G/K}$ does. 

 This observation applies to two cases of particular interest: Firstly, and trivially, we may always factor out normal subgroups $H$ contained in the kernel, and pass the action down to a representation of $G/H$. Moreover, if $\pi$ is irreducible and $Z < G$ is a central subgroup, then $\pi(Z) \subset \mathbb{T} \cdot {\rm id}_{\mathcal{H}_\pi}$ follows by Schur's lemma. Therefore $\pi$ does phase retrieval iff the associated projective representation of $G/Z$ does. 
\end{remark}

\begin{remark}
At this point, let us shortly explain why we focus on irreducible representations. This focus has a long tradition in representation theory, due to the fact that typically, an understanding of the irreducible cases is fundamental to an understanding of more complex representations. The general intuition is that a decomposition into irreducible subrepresentations allows to reduce the general problem into a family of subproblems, which can then be addressed independently. 

In the context of the phase retrieval, this intuition works to a limited degree. If $(\pi,\mathcal{H}_\pi)$ is a representation, and $\mathcal{H}_\pi$ decomposes into two closed invariant subspaces, $\mathcal{H}_\pi = \mathcal{K}_1 \oplus \mathcal{K}_2$, then any vector $\eta \in \mathcal{H}_\pi$ doing phase retrieval for $\pi$ gives rise to two vectors $\eta_1 \in \mathcal{K}_1,\eta_2 \in \mathcal{K}_2$ such that the restriction of $\pi$ to $\mathcal{K}_i$ does phase retrieval for $\mathcal{K}_i$, by taking $\eta_i$ to be the orthogonal projection of $\eta$ onto $\mathcal{K}_i$.

The standard representation-theoretic intuition would suggest that this observation also works for the converse direction, at least under suitable assumptions on the subrepresentations associated to the $\mathcal{K}_i$. Specifically, assuming that the subrepresentations are irreducible and inequivalent, one might be tempted to expect that a vector $\eta_i$ doing phase retrieval for each $\mathcal{K}_i$ gives rise to a vector $\eta = \eta_1 + \eta_2$ that does phase retrieval for $\mathcal{H}_{\pi}$. This is generally false, as the following example illustrates. 

Pick an irreducible representation $\sigma$ such that the contragredient representation $\overline{\sigma}$ is inequivalent to $\sigma$. Such representations exist, e.g. for the Heisenberg groups from Examples \ref{ex:finite_Heisenberg} (when $n>2$) and \ref{ex:Heisenberg_Lie}. Assume that $\eta_1 \in \mathcal{H}_\sigma$ is a vector doing phase retrieval for $\sigma$. If we then realize $\overline{\sigma}$ on $\mathcal{H}_\sigma$ in the canonical manner, the associated matrix coefficients are related by
\[
V_\eta^{\overline{\sigma}} f = \overline{V_\eta^\sigma f}~;
\]
this can be understood as the defining relation of the contragredient representation. 
On the one hand, this shows that $\eta$ also does phase retrieval for $\overline{\sigma}$. However, if we now define $\pi = \sigma \oplus \overline{\sigma}$, and use the vector $\eta_0 = \eta \oplus \eta$, we get 
\[
\left| V_{\eta_0}^\pi (f \oplus 0) \right| = \left| V^\sigma_\eta f \right| = \left| V_{\eta_0}^\pi (0 \oplus f) \right|~, 
\]
hence $\eta_0$ does not do phase retrieval. 

To summarize: Understanding the phase retrieval problem for irreducible representations will not lead to automatic solutions for arbitrary representations, but it remains a fundamental first step, that we address in the following for nilpotent groups. 
\end{remark}

We next need a representation-theoretic lemma that is quite probably folklore. It can be proved by direct calculations using the realization of induced representations via cross-sections, see e.g. page 79 of \cite{MR3012851}.

\begin{lemma}
 \label{lem:ind_char} Let $Z<N<G$ be locally compact groups, with $N$ normal and closed in $G$, and $Z$ central and closed in $G$. Let $\tau$ be a representation of $N$ with the property that the restriction $\tau|_Z$ is a multiple of a character $\chi \in \widehat{Z}$. Then the restriction $\left( {\rm ind}_N^G \tau \right)|_Z$ is also a multiple of $\chi$. 
\end{lemma}
%\begin{proof}
% We invoke the Mackey double coset theorem. For this purpose, let $S$ denote a Borel system of representatives of $N \backslash G / Z$. Under the assumptions given here, the double quotient simplifies to $G/N$. This implies in particular that $Z,N$ are regularly related subgroups (in the sense of \cite{MR44536}), and thus $S$ exists. Then Mackey's double coset theorem \cite[Theorem 12.2]{MR44536} yields 
% \[
%  \left( {\rm ind}_N^G \tau \right)|_Z \simeq \int^\oplus_{S} {\rm ind}_{s^{-1}Ns \cap Z}^Z (\tau^s) d\mu(s)~,
% \] where $\tau^s$ is the representation $g \mapsto \tau(sgs^{-1})$, restricted to $s^{-1}Ns \cap Z$, and $\mu$ is a suitably chosen measure on $S$. Under the given assumptions, we have $s^{-1} Ns = N \supset Z$, and $\tau^s = \tau|_Z$ is a multiple of $\chi$, for each $s \in S$. Hence each induced representation in the integral on the right is a multiple of $\chi$, as well.
%\end{proof}

We aim to prove the existence of vectors doing phase retrieval, for an irreducible representation $\pi$ of a simply connected, connected nilpotent Lie group $G$. The basic idea is to use induction over $\dim(G)$, and to employ Mackey's theory of induced representations in the process. For the induction step to work, we need a somewhat stronger induction hypothesis on the vectors doing phase retrieval. The relevant additional property is defined next:
\begin{definition}
Let $\pi$ denote a unitary representation of the locally compact group $G$. $\eta \in \mathcal{H}_\pi$ has the \textbf{full support property} if for all $0 \not= f \in \mathcal{H}_\pi$ the set
\[
\{ x \in G : V_\eta f (x) = 0 \} \subset G
\] has Haar measure zero. 
\end{definition}

\begin{remark} \label{rem:fs_factor_group}
 Just as for the case of phase retrieval in Remark \ref{rem:pr_factor_group}, the full support property can equivalently be verified for the associated projective representation $\overline{\pi}$ of $G/K$, where $K$ is any normal subgroup contained in the projective kernel of $\pi$. This observation makes use of the fact that the Haar measures of $G/K$ and $G$ are related by Weil's integral formula. 
\end{remark}

\begin{remark} \label{rem:fs_Heisenberg}
While the full support property has a strictly auxiliary role in our considerations, it seems of independent representation-theoretic interest. We are currently not aware of any sources investigating this property in any systematic way.  

As the prime example of a vector with the full support property, we mention the Gaussian window 
\[
\eta (x) = e^{-\pi x^2}~,(x \in \mathbb{R})~.
\]
It is well-known that the image of the short-time Fourier transform $STFT_\eta$ associated to $\eta$, defined by 
\[
STFT_\eta f (x+iy) = \langle f, M_yT_x \eta \rangle~,
\]
consists of functions whose product with a suitable, fixed choice of two-dimensional Gaussian is entire (see e.g. Proposition 3.4.1 of \cite{MR1843717}). But clearly $STFT_\eta f(x+iy) = V_\eta f(x,y,0)$, where the right-hand side denotes the matrix coefficient associated to the Schr\"odinger representation. Hence the fact that nonzero entire functions are nonzero a.e. (together with Remark \ref{rem:fs_factor_group}) allows to conclude that $\eta$ has the full support property for the Schr\"odinger representation. 

Note also that any vector $\eta$ with the full support property for the Schr\"odinger representation automatically does phase retrieval; recall from Example \ref{ex:Heisenberg_Lie} that a.e. nonvanishing of $V_\eta \eta$ was sufficient for the latter. Hence the Gaussian hass both desired properties.
\end{remark}

We are now in position to establish the first main result of this paper.
\begin{theorem} \label{thm:main_1}
 Let $G$ denote a simply connected, connected nilpotent Lie group, and $\pi$ an irreducible representation of $G$. Then there exists a vector $\eta \in \mathcal{H}_\pi$ that does phase retrieval, and has the full support property. In particular, $\pi$ does phase retrieval. 
\end{theorem}

\begin{proof}
 We proceed by induction over ${\rm dim}(G).$ Noting first that the abelian case is trivial, and secondly, that the case of the Heisenberg group is settled, we may assume that $G$ is a nonabelian group of dimension $n>3$. 
 
 Furthermore, we may assume that $K = {\rm Ker}(\pi)$ is a zero-dimensional subgroup of $G.$ Otherwise, the Lie algebra $\mathfrak{k}$ of $K$ has positive dimension, and the connected component of the identity $K_0 = \exp(\mathfrak{k})$ in $K$ is a normal subgroup of positive dimension.  Now the inductive hypothesis yields a vector with phase retrieval and full support property for the associated representation of the quotient group $G/K$, and thus for $\pi$ (see Remarks \ref{rem:pr_factor_group}, \ref{rem:fs_factor_group}.) 
 
 Hence from now on the kernel of $\pi$ is assumed to be zero-dimensional. This implies that the center $Z(G)$ of $G$ is necessarily one-dimensional.  Pick $Z \in \mathfrak{g}$ such that $\exp(\mathbb{R}Z) = Z(G)$. By Kirillov's lemma \cite[Lemma 1.1.12]{MR1070979}, there exist $X,Y \in \mathfrak{g}$ with the following properties: $\mathfrak{g} = \mathfrak{g}_0 \oplus \mathbb{R} X$, $[X,Y] = Z$, $\mathfrak{g}_0$ is an ideal in $\mathfrak{g}$, and the centralizer of $Y$. In particular, the associated subgroup $G_0 = \exp(\mathfrak{g}_0)$ is a normal subgroup, with $G = G_0 \exp(\mathbb{R}X)$. Furthermore, $A = \exp(\mathbb{R}Y) \exp(\mathbb{R} Z)$ is central in $G_0$, and normal in $G$. 
 
 For better readability, the remainder of the proof is structured into several steps.
 
 \noindent
 \textbf{Step 1: There exists an irreducible representation $\tau$ of $G_0$ such that $\pi \simeq {\rm ind}_{G_0}^G (\tau)$.}
 
 To see this, consider the dual action of $G$ on the dual group $\widehat{A}$, defined for arbitrary 
$g\in G$ and $a\in A,$ by $\left[  g\star\chi\right]  \left(  a\right)
=\chi\left(  g^{-1}ag\right).$ Let $G_{\chi}=\left\{  g\in G:g\star\chi
=\chi\right\}  $ be the stabilizer of $\chi$ and let
\[
\widetilde{G}_{\chi}=\left\{  \tau\in\widehat{G_{\chi}}:\tau
|_{A}\text{ is a multiple of }\chi\right\}  .
\]
For $\tau\in\widetilde{G}_{\chi},$ $\pi_{\tau}=\mathrm{ind}_{G_{\chi}%
}^{G}\left(  \tau\right)  $ is an irreducible representation of $G.$ In
fact, according to the theory of Mackey \cite[Theorem 4.22]{MR3012851}, every irreducible representation of
$G$ is of this form. Hence, since $\pi$ is assumed to be irreducible, there exists $\tau,$ such that $\pi\simeq\mathrm{ind}_{G_{\chi}}^{G} \left(  \tau\right)$ for suitable $\chi \in \widehat{A}$ and $\tau \in \widetilde{G}_{\chi}$. $\tau$ is irreducible, since $\pi$ is. To complete the proof of Step 1, it suffices to show that $G_\chi = G_0$. Since $A$ is central in $G_0$, the conjugation action of $G_0$ on $A$ is trivial. The same then holds for the dual action of $G_0$ on $\widehat{A}$, and this establishes $G_0 \subset G_{\chi}.$ Moreover, since $\pi$ is irreducible, $\pi|_{\exp(\mathbb{R} Z)}$ is the multiple of a character, and by assumptions on the kernel of $\pi$, this character is necessarily, nontrivial. Furthermore, Lemma \ref{lem:ind_char} implies that this character coincides with $\chi$ on $Z(G)$. Hence, by suitably normalizing $Z$ and $Y$, when necessary, we may assume without loss of generality that $\pi(\exp(z Z)) = e^{2 \pi i z}~.$ The relation $[X,Y] = Z$ then entails
\[
 \exp(-xX) \exp(yY) \exp(xX) = \exp(yY) \exp(-xyZ)~,
\]
and this leads to the following string of equalities:
\[
 (\exp(x X) \star \chi) (\exp(yY)) =  \chi( \exp(yY) \exp(-xyZ)) =  e^{-2 \pi i xy} \chi(\exp(yY))~.
\]
These observations imply that $\exp(xX) \star \chi = \chi$ can only hold for $x=0$. Since every $g \in G$ can be factored uniquely as $g =  \exp(x X) g_0$, with $g_0 \in G_0 \subset G_\chi$ and $x \in \mathbb{R}$, we derive finally that $G_0 = G_\chi$.

 \noindent
\textbf{Step 2: Choosing an explicit realization of $\pi$.}
We will now introduce coordinates on $G$, and express $\pi$ more explicitly in terms of these coordinates, and a suitable choice of realization. 
Thus far, we have shown that $\pi \simeq {\rm ind}_{G_0}^G (\tau)$, where $G_0 \lhd G$, and $\tau$ is an irreducible representation, restricting to a multiple of a character $\chi$ on $A = \exp(\mathbb{R}Y) \exp(\mathbb{R} Z)$. Furthermore, $\chi(\exp(zZ)) = e^{2 \pi i z}$. Let $\mathfrak{g}_0 \subset \mathfrak{g}$ denote the Lie algebra of $G_0$, and let $W \subset \mathfrak{g}_0$ denote a vector space complement to $\mathbb{R}X \oplus \mathbb{R} Y \subset \mathfrak{g}_0.$ Then every $g \in G$ has the unique factorization  
\[
 g = \exp(zZ)  \exp(yY)  \exp(w) \exp(xX)~,
\] with $x,y,z \in \mathbb{R}$ and $w \in W$. Consequently, every element in the closed and normal subgroup $G_0$ also admits a unique factorization of the type $g_0 = \exp(zZ) \exp(yY) \exp(w).$ 

As a consequence of the factorizations, we obtain that $G$ is the inner semidirect product $G = G_0 \rtimes \exp(\mathbb{R} X)$, which allows to use a specific realization of induced representations (see e.g. page 79 of \cite{MR3012851}, \textit{ Realization III for Semidirect Products}). The representation space can be taken as $\mathcal{H}_\pi = L^2(\mathbb{R},\mathcal{H}_\tau)$, the space of weakly Borel-measurable $\mathcal{H}_\tau$-valued functions on $\mathbb{R}$ with square-integrable norms, and the cited source provides the following explicit formulae
\begin{equation} \label{eqn:realize_pi} \left[ \pi \left( g\right) f\right] \left( t\right) =\begin{cases}f\left( t-x\right) \text{ if } g=\exp \left( x X\right) \\
e^{-2\pi ity}f\left( t\right) \text{ if }  g=\exp \left( yY\right) \\
\tau \left( \exp \left( e^{ad\left( -tX\right) }w\right) \right) f\left( t\right) \text{ if } g=\exp \left( w\right)~, w \in W \end{cases}.\end{equation}
It is worth noting that $\pi$ acts partially like a Schr\"odinger representation via $\exp(x X)$ and $\exp(y Y).$ 

 \noindent
\textbf{Step 3: Fixing the vector $\eta$.}
By Remark \ref{rem:fs_Heisenberg}, there exists $\eta_1 \in L^2(\mathbb{R})$ with full support and phase retrieval properties for the Schr\"odinger representation. 
Furthermore, the inductive hypothesis provides $\eta_2 \in \mathcal{H}_\tau$ with full support and phase retrieval properties for $\tau$. We then define $\eta \in L^2(\mathbb{R},\mathcal{H}_{\tau})$ by $\eta(t) = \eta_1(t) \eta_2$, and claim that it has the full support and phase retrieval properties for $\pi$.

\noindent
\textbf{Step 4: Explicit formulae for matrix coefficients.}
Fix any nonzero $f \in L^2(\mathbb{R},\mathcal{H}_\tau)$. By abuse of notation, we will systematically use
\[
V_\eta f(y,w,x) ~,~(y,w,x) \in \mathbb{R} \times W \times \mathbb{R}
\]
instead of 
\[
V_\eta(\exp(yY) \exp(w) \exp(xX))
\]
We will also omit the central variable $zZ$ from our discussion, in view of Remarks \ref{rem:pr_factor_group} and \ref{rem:fs_factor_group}. 

We then get from (\ref{eqn:realize_pi}) that
\begin{eqnarray*}
V_\eta f (y,w,x) & = & \langle f, \pi(\exp (yY) \exp (w) \exp(x X)) \eta \rangle \\ & = & \int _{\mathbb{R} }\langle f\left( t\right) ,\tau\left( \exp \left( e^{-ad\left( t X\right) }w\right) \right) \eta \left( t-x\right) \rangle e^{2 \pi i y t}dt \\
& = & \int _{\mathbb{R} }\langle f\left( t\right) ,\tau\left( \exp \left( e^{-ad\left( t X\right) }w\right) \right) \eta_2 \rangle \overline{e^{2 \pi i y t} \eta_1(t-x)} dt \\
& = & \int_{\mathbb{R} }\langle f\left( t\right) ,\tau\left( \exp(-t X) \exp(w) \exp(t X) \right) \eta_2 \rangle \overline{e^{2 \pi i y t} \eta_1(t-x)} dt ~.
\end{eqnarray*}
Now, using $\varrho$ to denote the Schr\"odinger representation, as well as
\begin{equation} \label{eqn:def_F}
F: W \times \mathbb{R} \to \mathbb{C}~,~ F(w,t) = \langle f\left( t\right) ,\tau\left( \exp(-t X) \exp(w) \exp(t X) \right) \eta_2 \rangle
\end{equation}
we may rewrite this formula as
\begin{equation} \label{eqn:V_eta_rho}
    (V_\eta f)(y,w,x) = \left(V_{\eta_1}^\varrho F(w,\cdot) \right) (y,x)~.
\end{equation}
In addition, we make the observation that 
\begin{equation} \label{eqn:F_mc}
F(w,t) = \left( V_{\eta_2}^\tau f(t) \right) \left( \exp(-t X) \exp(w) \exp(t X)  \right)~.
\end{equation}

 \noindent
\textbf{Step 5: $\eta$ does phase retrieval.}
Now assume that $f,h \in L^2(\mathbb{R},\mathcal{H}_\tau)$ are such that $|V_\eta f| = |V_\eta h|$ holds. With $F$ from (\ref{eqn:def_F}) and $H$ defined analogously as 
\[
H: W \times \mathbb{R} \to \mathbb{C}~,~ H(w,t) = \langle h\left( t\right) ,\tau\left( \exp(-t X) \exp(w) \exp(t X)  \right) \eta_2 \rangle
\]
equation (\ref{eqn:V_eta_rho}) gives for all $w \in W$
\[
 \left| V_{\eta_1}^\varrho (F(w,\cdot)) \right| =  \left| V_{\eta_1}^\varrho (H(w,\cdot)) \right| ~.
\] Since $\eta_1$ does phase retrieval, this implies the existence of a map 
\begin{equation} \label{eqn:defn_ph_u}
u: W \to \mathbb{T}~,~ H(w,\cdot) = u(w)F(w,\cdot)~~\forall w \in W. 
\end{equation} It follows for almost all $t \in \mathbb{R}$ that
\[
|H(\cdot,t)| = |F(\cdot,t)|~.
\] By construction, $\exp(W) \subset G_0$ is a system of representatives modulo the central subgroup $\exp(\mathbb{R}Z) \exp(\mathbb{R}Y)$ of $G_0$, which is contained in the projective kernel of $\tau$. Hence equation (\ref{eqn:F_mc}) (and its analog for $H$) allows to rewrite the previous equation as follows, for almost all $t \in \mathbb{R}$:  
\[
\forall g_0 \in G_0 ~:~ \left| V_{\eta_2}^\tau (f(t)) (\exp(-tX) g_0 \exp(tX)) \right| = \left| V_{\eta_2}^\tau (h(t)) (\exp(-tX) g_0 \exp(tX)) \right|~.
\] For fixed $t \in \mathbb{R}$, conjugation by $\exp(tX)$ is a bijection of $G_0$ onto itself, hence we may simplify this equation to 
\[
\left| V_{\eta_2}^\tau (f(t)) \right| \equiv \left| V_{\eta_2}^\tau (h(t)) \right|~.
\] Now we can appeal to the phase retrieval property of $\eta_2$ to conclude the existence of a map 
\begin{equation} \label{eqn:role_v_cont}
v : \mathbb{R} \to \mathbb{T}~,~h(t) = v(t) f(t)~ 
\end{equation} with the equality holding for almost every $t \in \mathbb{R}$. Note that this also entails, for almost every $t$, that 
\[
H(\cdot,t) = v(t) G(\cdot,t)~.
\]
Defining the Borel set 
\begin{equation} \label{eqn:def_C}
    C_f = \{ t \in \mathbb{R}~:~ f(t) \not= 0 \} \subset \mathbb{R}~,
\end{equation} (\ref{eqn:role_v_cont}) shows that the proof of phase retrieval boils down to showing that $v$ is constant on $C_f$. Note that by removing a set of measure zero from $C_f$ we can assume that (\ref{eqn:role_v_cont}) holds for all $t \in C_f$.

At this point the full support property of $\eta_2$ enters the picture. Picking any $t \in C_f$, equation (\ref{eqn:def_F}) and the full support property of $\eta_2$ implies that the set
\[
W_f(t) = \{ w \in W : F(w,t) \not= 0 \}
\] has full measure. Note that this step also used the fact that the automorphism of $G_0$ induced by conjugation with $\exp(tX)$ preserves the Haar measure of $G_0$, and Remark \ref{rem:fs_factor_group}. 

Hence, picking $t_1,t_2 \in C_f$ and any $w$ in $W_f(t_1) \cap W_f(t_2)$ (which is of full measure, hence nonempty), we obtain from (\ref{eqn:defn_ph_u})
\begin{eqnarray*}
v(t_i) F(w,t_i) = H(w,t_i) = u(w) F(w,t_i)~.\end{eqnarray*}
 By choice of $w$, we may cancel $F(w,t_i)$ on both sides, to obtain 
 $v(t_1) = u(w) = v(t_2)$. This concludes the proof of phase retrieval. 

\noindent
\textbf{Step 6: $\eta$ has the full support property.}
Let $0 \not= f \in L^2(\mathbb{R},\mathcal{H}_\tau)$ be given, and let $C_f$ be as in (\ref{eqn:def_C}). We will prove that the set 
\[
\left\{ (y,w,x) \in \mathbb{R} \times W \times \mathbb{R} : V_\eta f(y,w,x) = 0 \right\}
\] has measure zero. Then an appeal to Remark \ref{rem:fs_factor_group} finishes the proof.

Given any $w \in W$, equation (\ref{eqn:V_eta_rho}) together with the full support property of $\eta_1$ shows that 
\[
\left\{ (y,x) \in \mathbb{R}  \times \mathbb{R} : V_\eta f(y,w,x) = 0 \right\} 
\] has measure zero iff 
\[
\| F(w,\cdot)\|_{L^2(\mathbb{R})} \not= 0~,
\] with the function $F$ from (\ref{eqn:def_F}).

Thus by Fubini's theorem, the full support property for $\eta$ is shown once we establish that  
\[
\left\{ w \in W: \int_{\mathbb{R}} |F(w,t)|^2 dt = 0 \right\}
\] has measure zero. For this purpose, introduce the set
\[
B_f = \{ (w,t) \in W \times C_f : F(w,t) = 0 \}~,
\] with $C_f$ defined in Step 5. We also introduce the auxiliary function 
\[
\tilde{F}(w,t)= \left( V_{\eta_2}^\tau f(t) \right) \left( \exp (w) \right) 
\]
and associated set 
\[
\tilde{B}_f = \{ (w,t) \in W \times C_f :  \tilde{F}(w,t) = 0 \}~.
\]

Given $t \in \mathbb{R}$,
consider the map $\alpha(t): W \to W$ defined as 
\[
\alpha(t) (w) = P_W \left( e^{-ad(tX)} w \right)
\] where $P_W$ is the projection onto $W$ inside $\mathfrak{g}_0$, along the complement $\mathbb{R} Z \oplus \mathbb{R} Y$. Then by definition of $\alpha(t)$, we have 
\[
\exp(\alpha(t) w)^{-1} \exp(-t X) \exp(w) \exp(t X)  \in Z(G_0)~, 
\] and since $\tau$ restricted to $Z(G_0)$ is a character, this gives rise to the equation 
\[
|F(w,t)| = |\tilde{F}(\alpha(t)w,t)|~.
\]
We also note that $\alpha(t)$ is a diffeomorphism, with inverse map given by $\alpha(-t)$. %Furthermore, the fact that conjugation by $\exp(tX)$ leaves the Haar measure of $G_0$ invariant implies the invariance of Lebesgue measure on $W$ under $\alpha_t$. 

These observations allow to write 
\begin{equation} \label{eqn:Bf_vs_tilde}
B_f = \{ (w,t) \in W \times C_f : (\alpha(t)w, t) \in \tilde{B}_f \}~. 
\end{equation}

The whole point of introducing $\tilde{B}_f$ was that the assumption that $\eta_2$ has the full support property, and the definition of $C_f$, imply that $\tilde{B}_f$ is a set of measure zero, by Fubini's theorem. It then follows by (\ref{eqn:Bf_vs_tilde}), a change of variable involving the $\alpha(t)$, and yet another application of Fubini's theorem, that $B_f$ has measure zero. A final application of Fubini's theorem then yields a set $W' \subset W$ of full measure such that $F(w,\cdot)$ vanishes almost nowhere on $C_f$, for all $w \in W'$. Since $C_f$ is of positive measure, this finally implies
\[
\int_{\mathbb{R}} |F(w,t)|^2 dt>0
\] for all $w \in W'$, and the proof is finished. 
\end{proof}

\section{Finite nilpotent groups}

This section is devoted to establishing analogs of Theorem \ref{thm:main_1} for finite nilpotent groups. While we will not be able to cover all finite nilpotent groups, we will exhibit large classes of finite nilpotent groups, for which an analog can be formulated and proved, using an inductive method that is quite similar to that employed for Lie groups. 

The following theorem contains the largest previously known class of finite groups for which phase retrieval is established:
\begin{theorem} \label{thm:class_2}
Let $G$ denote a finite nilpotent group of nilpotency class 2, which means that $G$ is nonabelian, and $G/Z(G)$ is abelian. Then every irreducible representation of $G$ does phase retrieval. 
\end{theorem}
This theorem is a consequence of Theorem 1.7 from \cite{MR3901918}, which states that every projective irreducible representation of a finite abelian group does phase retrieval, combined with the natural correspondence between projective representations of finite abelian groups and unitary representations of class 2 nilpotent groups. 
Our results will significantly extend this class of examples.

Recall that a $p$-group is a group $G$ with cardinality $p^r$, for some positive exponent $r>0$ and prime number $p$. If $p^k$ is the highest power of some prime number $p$ dividing the order of $G,$ then a subgroup of $G$ of order $p^k$ is called a $p$-Sylow subgroup of $G.$ The \textit{order} of a group $G$ is the smallest integer $n$ satisfying $x^n = e$ for all $x \in G$. Clearly, the order of a $p$-group is again a power of $p$.

The treatment of general finite nilpotent groups is easily reduced to that of $p$-groups, by the following well-known observation (see e.g. Theorems 5.2.4 and 8.4.2 of \cite{MR648604}).
\begin{theorem} \label{p_group} Every finite nilpotent group $G$ is the direct product of its $p$-Sylow subgroups. As a consequence, every irreducible representation of $G$ is the outer tensor product of irreducible representations of its $p$-Sylow subgroups.  \end{theorem}

%We will now show that the outer tensor product of two representations with the phase retrieval property inherits said property. We first give a slightly more general formulation. 

The following Lemma is a direct consequence of Lemma \ref{lem:pr_HS}, which immediately gives a solution for the phase retrieval problem for tensor product representations in terms of solutions for the factors. Since it immediately generalizes to finitely many tensor factors, this lemma reduces the phase retrieval problem for irreducible representation of general finite nilpotent groups to the $p$-group case, on which we shall subsequently concentrate. 

\begin{lemma}
\label{PR copy(7)} For $k=1,2,$ let $\pi_{k}$ be a unitary representation of
some finite group $G_{k}$ and suppose that each $\pi_{k}$ does phase retrieval
and is realized is some Hilbert space $\mathcal{H}_{\pi_{k}}.$ Next, let
$\pi=\pi_{1}\otimes\pi_{2}$ be the outer tensor product of $\pi_1$ and $\pi_2$. If $\eta_{1}$ does phase retrieval for $\pi_{1}$ and
if $\eta_{2}$ does phase retrieval for $\pi_{2}$ then $\eta_{1}\otimes\eta
_{2}\in\mathcal{H}$ does phase retrieval for $\pi.$
\end{lemma}

\vskip0.5cm \noindent
%Throughout this subsection, we assume that $G$ is a $p$-group, for a prime number $p$. 

Before we turn to the study of phase retrieval properties for $p$-groups, we will need to establish several auxiliary notions and results. To a large part, these notions are motivated by the desire to adapt the inductive approach developed for the Lie group case to the $p$-group setting. They mostly concern the inductive construction of irreducible representations. But we will also need a finite group analog to the full support property, which is necessarily of a more quantitative nature. We start by formulating this analog: 

\begin{definition} \label{defn:p0}
Let $(\pi,\mathcal{H}_\pi)$ denote a unitary representation of a finite group $G$. Given a vector $\eta \in \mathcal{H}_\pi$, we let 
\[
p_0(\pi,\eta) = \max_{f \in \mathcal{H}_\pi \setminus \{ 0 \}} \left( 1- \frac{|{\rm supp}(V_\eta f)|}{|G|} \right) 
\] as well as 
\[
p_0(\pi) = \min \{ p_0(\pi,\eta) : \eta \in \mathcal{H}_\pi \setminus \{ 0 \} \}~. 
\]
\end{definition}

The following remark formulates an analog to Remark \ref{rem:pr_factor_group} for the quantity $p_0$:
\begin{remark} \label{rem:p_0_proj_kernel}
 Let $\pi$ be a representation of a finite group $G$, and $K \lhd G$ is a normal subgroup contained in the projective kernel of $\pi$. Picking any system $W \subset G$ of representatives modulo $K$, it is easy to see that 
 \[
p_0(\pi,\eta) = \max_{f \in \mathcal{H}_\pi \setminus \{ 0 \}} \left(1- \frac{\left|{\rm supp}\left((V_\eta f)|_W\right) \right|}{|W|} \right)
\]
\end{remark}

\begin{remark} \label{rem:p_0}
 $p_0(\pi,\eta)$ can be interpreted as the maximal proportion of zeros occurring in the matrix coefficients $V_\eta f$ for an arbitrary nonzero vector $f$. This quantity is therefore related to sampling sets: For any set $A \subset G$ satisfying $|A| > p_0(\pi,\eta) |G|$, the sampled coefficient map
 \[
 V_\eta|_A : f \mapsto (V_\eta f)|_A
 \] has trivial kernel by definition of $p_0(\pi,\eta)$, hence it is injective. 
 
 This property is also related to the full spark property, studied in the representation theoretic context in \cite{MR4121508,MR3303677}. In fact, for any representation $\pi$ having the full spark property, one can derive 
 \[
 p_0(\pi) = \frac{d_\pi -1}{|G|}~.
 \] 
 
 A related result was derived by Malikiosis in \cite{MR3303677}, who showed that, for suitable choices of $\eta \in \mathbb{C}^n$ the system $(\pi(k,l,0) \eta)_{k,l \in \mathbb{Z}_n}$ has the full spark property, where $\pi$ is the Schr\"odinger representation from Example \ref{ex:finite_Heisenberg}, associated to the finite cyclic group $\mathbb{Z}_n$ of order $n$. Since $W = \{(k,l,0): k,l = 0,\ldots, n-1 \} \subset \mathbb{H}_n$ is a system of representatives modulo the center of $\mathbb{H}$, which is contained in the projective kernel of $\pi$, Remark \ref{rem:p_0_proj_kernel} allows to determine  
 \begin{equation} \label{eqn:p_0_Schroedinger}
 p_0(\pi) = \frac{n-1}{n^2}~.
 \end{equation} 
\end{remark}

\begin{remark} \label{rem:p_0_vs_pr}
In our subsequent arguments (see Lemma \ref{lem:pr_inductive} below), an upper estimate for the quantity $p_0(\tau)$ for a certain representation of a suitably chosen subgroup $G_0$ of $G$ will play an important role, as a sufficient condition for phase retrieval. It is noteworthy that the phase retrieval property of $\pi$ also entails a necessary lower estimate for $p_0(\pi)$, as follows: Assuming that the vector $\eta \in \mathcal{H}_\pi$ does phase retrieval for $\pi$, Lemma \ref{lem:basics_pr}(b) implies that for any nonzero $f,g \in \mathcal{H}_\pi$, the pointwise product $(V_\eta f) \cdot (V_\eta g)$ is not identically zero. Letting $A = {\rm supp}(V_\eta f) \subset G$ and $g = \pi(y) f$ for $y \in G$ yields via
\[
V_\eta g (x) = \langle \pi(y) f, \pi(x) \eta \rangle = V_\eta f (y^{-1}x)
\]
that 
$yA \cap A \not= \emptyset$, or equivalently: $y \in A \cdot A^{-1}$. Hence we have shown 
\[
A \cdot A^{-1} = G~,
\] for the support $A$ of any nonzero matrix coefficient $V_\eta f$. This observation, together with Remark \ref{rem:p_0}, allows to derive the estimate
\[
p_0(\pi,\eta) \le 1-|G/K|^{-1/2}~,
\]
where $K \lhd G$ denotes the projective kernel of $\pi$. 
\end{remark}

For the formulation of the next lemma, we introduce a new piece of terminology: Given a finite dimensional Hilbert space $\mathcal{H}$, we call a set $M \subset \mathcal{H}$ \textbf{real Zariski closed} if there exist mappings 
\[
f_1,\ldots,f_k : \mathcal{H} \to \mathbb{R}
\]
that depend polynomially on the real and imaginary parts of arbitrarily chosen complex linear coordinates on $\mathcal{H}$, and such that 
\[
M = \{ x\in \mathcal{H}~:~ f_1(x) = \ldots = f_k(x) = 0 \}~.
\]
We call $M$ \textbf{real Zariski open} if its complement is real Zariski closed.

\begin{lemma} \label{lem:Zariski}
Let $(\pi,\mathcal{H}_\pi)$ denote a finite-dimensional representation of the finite group $G$. Let
\[
M_1 = \{ \eta \in \mathcal{H}_\pi : \eta \mbox{ does phase retrieval } \}
\]
and 
\[
M_2 = \{ \eta \in \mathcal{H}_\pi : p_0(\pi,\eta) = p_0(\pi) \}~.
\]
Then $M_1$ and $M_2$ are real Zariski open. If $\pi$ does phase retrieval, then $M_1 \cap M_2$ is nonempty and real Zariski open. In particular $M_1 \cap M_2 \subset \mathcal{H}_\pi$ is of full Lebesgue measure. 
\end{lemma}

\begin{proof}
For the statement concerning $M_1$ we refer to Lemma 2.5 of \cite{bartuselphase}. 
For $M_2$, we identify $\mathcal{H}_\pi$ with $\mathbb{C}^{d_\pi}$ (using any choice of basis in $\mathcal{H}_\pi$), identify $\mathbb{C}^G$ with $\mathbb{C}^{|G|}$ by a similar procedure, and use these identifications to associate to the linear operator $V_\eta: \mathbb{C}^{d_\pi} \to \mathbb{C}^G$ the describing matrix $\mathcal{A}_\eta \in \mathbb{C}^{|G| \times d_\pi}$. Given any matrix $\mathcal{A} \in \mathbb{C}^{|G|  \times d_\pi}$, and an arbitrary set $B \subset \{1,\ldots,|G|\}$, we let $\mathcal{A}^B$ denote the submatrix obtained by picking the lines of $\mathcal{A}$ corresponding to indices contained in $B$. Then 
\[
p_0(\pi,\eta) \le \frac{k}{|G|}
\] is equivalent to the statement that, for arbitrary sets $B \subset \{1,\ldots,|G|\}$ of cardinality $k+1$, the submatrix $\mathcal{A}_{\eta}^B$ has rank $d_\pi$. But the set
\[
\mathcal{Z}_{k} = \{ \mathcal{A} \in  \mathbb{C}^{|G| \times d_\pi} : \forall B \subset \{1,\ldots,|G| \} \mbox{ with } |B| = k+1 ~,~ {\rm rank}(\mathcal{A}^B) = d_\pi \} 
\] can be described in terms of $d_\pi \times d_\pi$-subdeterminants, hence is real Zariski-open. Now we have that
\begin{eqnarray*}
p_0(\pi,\eta) = p_0(\pi) & \Leftrightarrow & p_0(\pi,\eta) \le p_0(\pi) \\
 & \Leftrightarrow & \mathcal{A}_\eta \in \mathcal{Z}_{\pi_0(\pi) |G|}~.
\end{eqnarray*}
Since the map $\eta \mapsto A_\eta$ is conjugate linear, we see that $M_2$ is conjugate linear preimage of a real Zariski open set, hence real Zariski open itself. 
Now if $\pi$ does phase retrieval, then $M_1$ is nonempty, whereas $M_2$ is nonempty by definition. Nonempty real Zariski open sets are open and dense in the standard topology; in fact they are of full Lebesgue measure. Then $M_1 \cap M_2$ is nonempty, and real Zariski-open, hence it inherits the remaining properties as well. 
\end{proof}

We will now start to introduce the central device of our treatment of $p$-groups, namely a version of Kirillov's lemma for this class of groups. For this purpose, we introduce further notation. Let $G$ denote a nonabelian $p$-group. Let $[w,y] = w^{-1} y^{-1} w y$ be the commutator of $w,y \in G$. Furthermore, write $w^y = y^{-1} w y$ for the conjugation product of $w$ and $y$. We then have the well-known relations
\begin{equation}\label{eqn:comm_rel}
 [xz,y] = [x,y]^z [z,y]~,~[x,zy] = [x,y] [x,z]^y ~,
\end{equation} which can be verified by direct calculation. 

We introduce 
\[
 Z_1(G) = \{ y \in G: y Z(G) \in Z(G/Z(G)) \}~.
\] Since $G$ is nilpotent, $G/Z(G)$ is as well. In particular, $G/Z(G)$ has nontrivial center, which implies that $Z(G) \subsetneq Z_1(G)$, unless $G$ is abelian.  
The following elementary lemma will be used repeatedly: 
\begin{lemma} \label{lem:bihom}
The mapping 
\[
 [ \cdot, \cdot ] : G \times Z_1(G) \to Z(G)~,~(x,y) \mapsto [x,y]
\] is a well-defined bihomomorphism, i.e., for all $x \in G, y \in Z_1(G)$, the maps $[x,\cdot]$ and $[\cdot,y]$ are homomorphisms into $Z(G)$. 
\end{lemma}
\begin{proof}
Let $x \in G, y \in Z_1(G)$. 
  The fact that $y Z(G)$ is central in $G/Z(G)$ implies 
  \[
   [x,y] Z(G) = [x Z(G), y Z(G)] = e_G Z(G)~,
  \]
 hence $[x,y] \in Z(G)$, and $[\cdot, \cdot] : G \times Z_1(G) \to Z(G)$ is well-defined. Furthermore, (\ref{eqn:comm_rel}) yields for $x,z \in G$  
\[
  [xz,y] = [x,y]^z [z,y] = [x,y] [z,y]~,
\] since the conjugation action is trivial on $[x,y] \in Z(G)$. 
If $z \in Z_1(G)$, we similarly get
\[
 [x,zy] = [x,y] [x,z]^y = [x,y] [x,z] = [x,z] [x,y]~,
\] since $Z(G)$ is abelian. 
\end{proof}

The following is the announced $p$-group analog of Kirillov's lemma for nilpotent Lie groups \cite[Lemma 1.1.12]{MR1070979}. 
\begin{lemma}
\label{Kirillov} Let $G$ be a
non-abelian $p$-group with cyclic center $Z\left(  G\right)  $.
Then there exist $z \in Z(G), y \in Z_1(G)$ and $x \in G$ as well as an integer $r \ge 1$ such that the following hold:
\begin{enumerate}
 \item[(a)] $z = [x,y]$, $\langle z \rangle$ is a subgroup of $Z(G$) with $\langle z \rangle = [G,y] \cong \mathbb{Z}_{p^r}$.
 \item[(b)] $A = \langle y, Z(G) \rangle$ is an abelian normal subgroup of $G$ with $A/Z(G) = \langle y Z(G) \rangle \cong \mathbb{Z}_{p^r}$.
 \item[(c)] Define $G_0 = \{ w : [w,y] = e_G \}$. Then $G_0$ is a normal subgroup of $G$ with $G/G_0 = \langle x G_0 \rangle \cong \mathbb{Z}_{p^r}$. $A < G_0$ is central. 
\end{enumerate}
\end{lemma}

\begin{proof}
Proof of Part (a): Pick any $y \in Z_1(G) \setminus Z(G)$. By Lemma \ref{lem:bihom}, $[G,y]$ is a subgroup contained in the center $Z(G).$  Since the center of $G$ is assumed to be cyclic, $[G,y]$ is cyclic as well, hence $[G,y] \cong \mathbb{Z}_{p^r}$ for some natural number $r$, since $G$ is a $p$-group. The case $r=0$ is equivalent to $[G,y] = \{ e \}$, or $y \in Z(G)$, contrary to the choice of $y$. The existence of $z$ and $x$ is then obvious. 

Proof of Part (b): It is clear that $A$ is an abelian subgroup of $G$. Fixing $w \in G$, and appealing to Lemma \ref{lem:bihom}, we obtain $$ w^{-1} y w = [w,y^{-1}] y = \underbrace{[w,y]^{-1}}_{\in Z(G)} y \in A~. $$  This implies that $\langle y, Z(G) \rangle$ is normal. To complete the proof of Part (b), we show that $y Z(G)$ has order $p^r$ in $G/Z(G)$. In order to see this, we use Lemma \ref{lem:bihom} to compute $[x,y^l] = [x,y]^l = z^l \not= e_G$  for $1 \le l < p^r$, by choice of $z$. This shows in particular $y^l \not\in Z(G)$. 

On the other hand, assuming $y^{p^r} \not\in Z(G)$ would entail the existence of $w \in G$ with $ e_G \not= [w,y^{p^r}] = [w,y]^{p^r} ~.$ This would imply that  $ [w,y] \in [G,y]$ has order $>p^r$, which contradicts (a). 

Proof of Part (c): Note that by definition, $G_0$ is the kernel of the homomorphism $[\cdot,y]:G \to Z(G)$, which has image $[G,y]$, generated by $z = [x,y]$. Hence $G_0$ is a normal subgroup, and $\langle xG_0 \rangle = G/G_0 \cong \langle z \rangle \cong \mathbb{Z}_{p^r}$ follows by the first homomorphism theorem. $y$ is central in $G_0$ by definition of the latter, and $Z(G) \subset G_0$ is central even in $G$, hence $A = \langle y, Z(G) \rangle$ is central in $G_0$. 
\end{proof}

\begin{remark}
 We note that the case $[G,y] \subsetneq Z(G)$ may well occur. For a concrete example, fix a prime number $p$ and consider
 $ G = \mathbb{Z}_p \times \mathbb{Z}_p \times \mathbb{Z}_{p^2}$,
 with group law
 \[
  (k,l,m)\cdot(k',l',m') = (k+k',l+l',m+m'-pl'k)~.
 \] 
 The center of this group is given by $\{ 0 \} \times \{ 0 \} \times \mathbb{Z}_{p^2}$, whereas 
 $[G,y] = \langle (0,0,p) \rangle$. 
\end{remark}

The notation of the lemma suggests that the group $H = \langle x,y, z \rangle$, for $x,y,z$ from Lemma \ref{Kirillov} is isomorphic to a Heisenberg group. Note that the analogous statement in the Lie group case, i.e. that the closed subgroup associated to the Lie subalgebra generated by $X,Y,Z$ is isomorphic to the three-dimensional Heisenberg group, holds true. The next example shows however that the $p$-group case is more diverse. 
Nevertheless the Schr\"odinger representations from Example \ref{ex:finite_Heisenberg} constitute important basic building blocks in our inductive strategy to establish phase retrieval for general $p$-groups, and the following example is intended to shed some light on this phenomenon. 
%\textcolor{red}{ The interested reader might find a paragraph motivating this example, useful.}
\begin{example} \label{ex:finite_Kirillov} 
Let $H = \mathbb{H}_{p^r}$ denote a finite Heisenberg group associated to the cyclic group $\mathbb{Z}_{p^r}$, as given by (\ref{eqn:fH_set}) and (\ref{eqn:fH_law}). Define $x=(1,0,0),y = (0,1,0)$ and $z = (0,0,-1)$. Then one has 
$Z(H) = \langle z \rangle = Z_0$, 
\[
 Z_1(H) = H = \langle x,y, z \rangle = \langle x,y \rangle~,
\] and $[x,y] = z$, as well as all the other properties from Lemma \ref{Kirillov}.  
 It is tempting to assume that the relations in this lemma already characterize $H$ up to isomorphism.

 But this is not the case. For the construction of a counterexample, fix a prime number $p$ and an integer $r \ge 1$. 
 Consider the group $G = \mathbb{Z}_{p^{2r}} \times \mathbb{Z}_{p^r}$, with group law 
 \begin{equation} \label{eqn:group_ex_fK}
  (k,l)(k',l') = (k+(1-lp^r)k',l+l')~,0 \le k,k' < p^{2r}~, 0 \le l,l' < p^r~.,
 \end{equation} where the operations on the right-hand side are taken modulo $p^{2r}, p^r$, respectively. Using
 \[
  (1-lp^r)(1-l'p^r) \equiv 1-(l+l')p^r ~~\mbox{ mod } p^{2r} ~,
 \] one easily verifies that (\ref{eqn:group_ex_fK}) is a well-defined group law. 
 We let $x = (0,1),y=(1,0)$ and $z = (p^r,0)$, and obtain that $Z(G) = \langle z \rangle = Z_0$, 
 \[
  Z_1(G) = G = \langle x,y,z \rangle = \langle x,y \rangle = \langle y \rangle \rtimes \langle x \rangle~, 
 \] as well as $[x,y] = z$.
 
 Again, all remaining properties noted in Lemma \ref{Kirillov} hold for $x,y,z$ also. 
Nonetheless, $G$ is clearly not isomorphic to the Heisenberg group $H$, since it contains elements of order $p^{2r}$. 

In order to exhibit similarities of the newly constructed group to the Heisenberg group case, we now determine certain irreducible unitary representations of $G$. Letting $A = \langle y,z \rangle = \mathbb{Z}_{p^{2r}} \times \{ 0 \} \cong \mathbb{Z}_{p^{2r}}$, we let $\xi \in \mathbb{T}$ be a root of unity of order $p^{2r}$, and define the associated character $\chi_\xi \in \widehat{A}$ as 
\[
 \chi_\xi (k,0) = \xi^k ~,~k=0,\ldots,p^{2r}-1~.
\] The dual action of $x$ on the characters is given by 
\[
 x^l \ast \chi_\xi(k,0) = \chi_\xi ((0,-l)(k,0)(0,l)) = \xi^{k (1+lp^r)} = \chi_{\xi^{1+lp^r}} (k,0)~.
\]
Now let $\xi_0$ denote an arbitrary \textit{primitive root of order $p^{2r}$}. Then we have for $0 \le l < p^r$ that 
\[
 \xi_0^{(1+lp^r)} = \xi_0 \Longleftrightarrow \xi_0^{lp^r} = 1 \Longleftrightarrow l = 0 ~,
\] i.e., the stabilizer of $\xi_0$ under the dual action is trivial. 

We now consider $\pi = {\rm ind}_A^G \chi_{\xi_0}$. By Mackey's theorem, $\pi$ is an irreducible representation, and since $\langle y,z \rangle$ is a normal subgroup, we can realize 
$\pi$ on $l^2(G/A) = l^2(\langle x \rangle)$ by the formula  from page 79 of \cite{MR3012851}, which yields
\[
 \left[ \pi(y^k x^l) \varphi \right](x^j) = \chi_{\xi_0}(x^{-j} y^k x^j) \varphi(x^{j-l})~.
\]
Since $\left[  x,y\right]  =z,$
$x^{-j}\left(  y^{k}\right)  x^{j}=y^{k}z^{-kj}$, and we obtain
\begin{eqnarray*}
 \left[ \pi (y^k x^l) \varphi \right] (x^j) & = & \chi_{\xi_0} (y^{k}z^{-kj}) \varphi(x^{j-l}) \\
 & = &  \xi_0^{k} \xi_1^{-kj} \varphi(x^{j-l})  
\end{eqnarray*} where $\xi_1 = \xi_0^{p^r}$ is a \textit{primitive root of order $p^r$}. This observation may be rephrased as 
\[
 \pi(y^k x^l) = \alpha(k,l) M_k T_l~,
\] with $\alpha(k,l) \in \mathbb{T}$. This shows that, even though $G$ and $H$ are clearly non-isomorphic, the operators in the realization of $\pi$ differ from a Schr\"odinger representation only by scalar factors, and Lemma \ref{lem:basics_pr} (c) entails that the vectors doing phase retrieval for $\pi$ are precisely the ones that do phase retrieval for the Schr\"odinger representation. %This observation, made in a slightly different context, will be exploited below in Lemma \ref{lem:pr_finite_factor}.
\end{example}

The next definition introduces a bookkeeping device, intended to describe the construction of an arbitrary irreducible representations by repeated induction steps. 

\begin{definition} \label{defn:K-triple}
Let $G$ denote a $p$-group, and $\pi$ an irreducible, faithful representation of $G$. $(G_0,r,\tau)$ is called a \textbf{K-triple for $\pi$} if $r \in \mathbb{N}$, $G_0 \lhd G$, $\tau$ is a representation of $G_0$ such that $\pi \simeq {\rm ind}_{G_0}^G \tau$, and there exist $x \in G$, $y \in Z_1(G) \setminus Z(G)$ and $z \in Z(G)$ with the following properties: 
\begin{enumerate}
    \item[(i)] $[x,y] = z$, and $\langle z \rangle = [G,y] \cong \mathbb{Z}_{p^r}$.
    \item[(ii)] $G_0$ is the centralizer of $y$.
    \item[(iii)] $G/G_0 = \langle xG_0 \rangle \cong \mathbb{Z}_{p^r}$. 
\end{enumerate}
The K-triple is called \textit{split}, if $x$ can be chosen such that $x^{p^r} = e$. 
\end{definition}

It is easy to see that a K-triple $(G_0,r,\tau)$ is split if and only if the short exact sequence 
\[
 1 \to G_0 \to G \to G/G_0 \to 1 
\] splits, i.e. when $G$ is the inner semidirect product $G_0 \rtimes \langle x' \rangle$ for a suitable choice of $x'$. Here the ``only if''-part is clear. For the converse, assume that a K-triple $(G_0,r,\tau)$ is given, and that in addition $G = G_0 \rtimes \langle x' \rangle$ for a suitable choice of $x'$. This entails in particular that $x'$ has order $p^r$. Since $G_0$ is by definition the kernel of the homomorphism $[\cdot,y]: G \to \langle z \rangle$, we get that $s \mapsto [(x')^s,y] \in \langle z \rangle$ is onto, hence $z' = [x',y]$ fulfills $\langle z' \rangle = \langle z \rangle$. This shows that $x',y,z' \in G$ have the properties required of $x,y,z$ in Definition \ref{defn:K-triple} to show that $(G_0,r,\tau)$ is a split K-triple.

The following lemma provides a central result for our inductive approach. It is the $p$-group analog to Step 1 in the proof of Theorem \ref{thm:main_1}, and its proof is remarkably similar to the proof for the Lie group setting. 

\begin{lemma}
\label{lem:exists_K_triple} Let $G$ denote a nonabelian $p$-group, and $\pi$ a faithful irreducible representation of $G$. Then there exists a K-triple for $\pi$.
\end{lemma}

\begin{proof}
Since $\pi$ is irreducible, $\pi(Z(G))$ is a (finite) subgroup of the torus. It follows that $\pi(Z(G))$ is cyclic, and by faithfulness of $\pi$, this entails that $Z(G)$ is cyclic. Hence Lemma \ref{Kirillov} is applicable, and we obtain the group elements $x,y,z$, as well as the normal subgroup $G_0$ and $r \in \mathbb{N}$ such that the properties (i)-(iii) from Definition \ref{defn:K-triple} are fulfilled. It therefore remains to prove the existence of $\tau$. Note that $G_0$ is the centralizer of $y$ by its definition in part (c) of the lemma. 

Let $A = \langle y,Z(G) \rangle$. Then $A$ is an abelian normal subgroup of $G$. Note that this fact can be concluded from the choice according to Lemma \ref{Kirillov}, but it also follows from $z = [x,y] \in Z(G)$ and the choice of $G_0$ as centralizer of $y$. In a similar way, $A$ is seen to be central in $G_0$.

For any given unitary character $\chi$ of $A,$ $G$
acts on $\chi$ via the dual action which is defined as follows. For arbitrary
$g\in G$ and $a\in A,$ $\left[  g\star\chi\right]  \left(  a\right)
=\chi\left(  g^{-1}ag\right).$ Let $G_{\chi}=\left\{  g\in G:g\star\chi
=\chi\right\}  $ be the stabilizer of $\chi$ and let
\[
\widetilde{G}_{\chi}=\left\{  \tau\in\widehat{G_{\chi}}:\tau
|_{A}\text{ is a multiple of }\chi\right\}  ,
\] where $\widehat{G}_\chi$ denotes the set of (equivalence classes of) irreducible representations of $G_\chi$. 
For $\tau\in\widetilde{G}_{\chi},$ $\pi_{\tau}=\mathrm{ind}_{G_{\chi}%
}^{G}\left(  \tau\right)  $ is an irreducible representation of $G.$ In
fact, according to the theory of Mackey \cite[Theorem 4.22]{MR3012851}, every irreducible representation of
$G$ is of this form. 

Hence $\pi\simeq\mathrm{ind}_{G_{\chi}}^{G} \left(  \tau\right)  $ for suitable $\chi \in \widehat{A}$ and $\tau \in \widetilde{G}_{\chi}$. The proof is finished when we have shown that $G_\chi = G_0$. By condition (ii), $A$ is central in $G_0$, which means that the conjugation action of $G_0$ on $A$ is trivial. The same then follows for the dual action on $\widehat{A}$, which establishes $G_0 \subset G_{\chi}$.  

In order to show the reverse inclusion, we take a closer look at the dual action of $x$ and its powers. By Lemma \ref{lem:ind_char} applied with $Z = Z(G)$ and $N= G_0$, the characters obtained by restricting either $\pi$ or $\tau$ to $Z(G)$ must coincide. By assumption, $\pi$ is faithful, hence $\tau|_{Z_0}$ is faithful. Since $Z_0 = \langle z \rangle$, we get that $\tau(z^k) = \xi^k$, for a suitable $p^r$th root of unity $\xi \in \mathbb{T}$. In fact, faithfulness requires that $\xi$ is a \textit{primitive root of unity of order $p^r$}, i.e. $\xi^k = 1$ holds iff $p^r$ divides $k$, for all $k \in \mathbb{Z}$. 

Since $y \in Z_1(G)$, we can appeal to Lemma \ref{lem:bihom} to compute $x^{-m} y x^m =  [x^m,y^{-1}] y = z^{- m} y$ which leads to 
\[
 (x^m \star \chi)(y) = \chi (z^{-m}) \chi(y)  = \xi^{-m} \chi(y)~. 
\] Hence $x^m \star \chi = \chi$ entails $\xi^{-m} = 1$, and since $\xi$ is a primitive root of order $p^r$, this only happens for $m \in p^r \mathbb{Z}$. 

By (iii), every $g \in G$ can be factored uniquely as $g =  x^m g_0$, with $g_0 \in G_0 \subset G_\chi$ and $0 \le m < p^r$. This implies that 
\[
\chi =  g \star \chi = x^m \star \chi 
\] holds iff $m=0$, or equivalently, iff $g \in G_0$. This proves $G_0 = G_\chi$ and thus $\pi \simeq {\rm ind}_{G_0}^G \tau$. 
\end{proof}%

%TCIMACRO{\TeXButton{TeX field}{\vskip0.5cm \noindent} }%
%BeginExpansion
\vskip0.5cm \noindent
%EndExpansion

We next provide an explicit description of the representation $\pi$ in terms of the data provided by a K-triple. It is again quite analogous to the Lie group case, but this time the analogy requires an additional assumption (which is automatic in the Lie group setting), namely that the group extension $1 \to G_0 \to G \to G/G_0 \to 1$ splits. 

\begin{lemma} \label{lem:ind_rep_explicit}
Let $G$ denote a $p$-group, and $\pi$ an irreducible, faithful representation of $G$, with a split K-triple $(G_0,r,\tau)$. Let $x,y,z \in G$ denote the group elements associated to the triple according to Definition \ref{defn:K-triple}, with the property $x^{p^r} = e$. 
\begin{enumerate}
    \item[(a)] Let $W \subset G_0$ denote a system of representatives modulo $\langle \{ y \} \cup Z(G) \rangle$. Then every element $g \in G$ has a unique factorization
    \[
    g = z' y^k w x^l ~,~ z' \in Z(G),~k,l \in \{ 0,\ldots p^r-1 \}~, w \in W~.
    \]
    \item[(b)] The representation ${\rm ind}_{G_0}^ G \tau$ acts in the Hilbert space $l^{2}\left(  \mathbb{Z}_{p^r},\mathcal{H}_\tau \right)$ by the formula 
    \[
    \left[ \pi(y^k w x^l) \varphi\right] (j) = \xi_1^k \xi_2^{-kj} \tau(x^{-j} w x^j) \varphi(j-l)~, 
    \] where $\xi_1 \in \mathbb{T}$ is a suitable root of unity, and $\xi_2$ is a primitive root of unity of order $p^r$.
\end{enumerate}
\end{lemma}

\begin{proof}
Part (a) follows from properties (i)-(iii), and the choice of $W$.

For part (b), we use the semidirect product structure $G = G_0 \rtimes \langle x \rangle$, and again appeal to the realization of induced representations for semidirect products, elaborated on \cite[Page 79]{MR3012851}. We identify the quotient $G/G_0$ with the cyclic group $\mathbb{Z}_{p^r}$, using the identification $x^j G_0 \mapsto j$. 
Letting $w\in W$ and $\varphi \in l^{2}\left(  \mathbb{Z}_{p^r},\mathcal{H}_{\tau}\right)  $, we get by the cited result that 
\[
\left[  \pi\left(  y^{k}wx^{l}\right)  \varphi\right]  \left(  j \right)
=\tau\left(  x^{-j} \left( y^{k}w \right) x^{j}\right)  \varphi\left(
{j-l}\right)
\]
for $\left(  k,w,l\right)  \in \mathbb{Z}_{p^r}  \times
W\times  \mathbb{Z}_{p^r}$, $j \in \mathbb{Z}_{p^r}$. Since $\left[  x,y\right]  =z,$ we get
$x^{-j}  y^{k}   x^{j}=y^{k}z^{-kj}$ and
\begin{eqnarray} \nonumber
\pi\left(  y^{k}wx^{l}\right)  \varphi\left(  x^{j}\right)   &  = & \tau\left(
x^{-j} y^{k}  x^{j}\right)  \circ\tau\left(  x^{-j}
w   x^{j}\right)  \varphi\left(  j-l\right)  \\  \nonumber
&  = & \chi\left(  y^{k}z^{-kj}\right)  \cdot\tau\left(  x^{-j}\left(  w\right)
x^{j}\right)  \varphi\left(  j-l\right)  \\ \label{eqn:real_tau}
&  = & \chi\left(  y^{k}\right)  \cdot\chi\left(  z^{-kj}\right)  \cdot
\tau\left(  x^{-j}  w  x^{j}\right)  \varphi\left(
j-l \right)  .
\end{eqnarray}
Here $\chi : A \to \mathbb{T}$ denotes the character obtained by restricting the irreducible representation $\tau$ to the central subgroup $A$. Since $\pi$ is faithful, Lemma \ref{lem:ind_char} implies that $\tau$ is faithful on $Z(G)$. Hence, as $z$ has order $p^r$, there exists a primitive root of unity $\xi_2$ of order $p^r$ such that
\[
\tau(z^l) = \xi_2^{l}~.
\] Furthermore $\chi(y^k) = \xi_1^k$ for a suitable root of unity, whose order is at most the order of $y$. 
\end{proof}

\begin{lemma} \label{lem:conj_action}
Let $G$ denote a $p$-group, and $\pi$ an irreducible, faithful representation of $G$, with K-triple $(G_0,r,\tau)$. Let $W \subset G_0$ denote a system of representatives modulo $\langle \left\{y\right\}\cup Z(G) \rangle$, where $y\in Z_1(G)$ denotes the group element associated to the triple according to Definition \ref{defn:K-triple}. Then there exist mappings $\alpha_j : W \to W$ and $\beta_j : W \to \langle \left\{y\right\}\cup Z(G) \rangle$, $j=0,\ldots,p^r-1$ such that 
\[
\forall j = 0,\ldots,p^{r}-1,~\forall w \in W ~:~ x^{-j}wx^j = \beta_j(w) \alpha_j(w)~. 
\]
\end{lemma}
\begin{proof} 
Since $\langle y,z \rangle$ is normal in $G$, it is invariant under conjugation with $x^j$. Hence this conjugation induces an automorphism of the quotient group $G_0/\langle \left\{y\right\}\cup Z(G) \rangle$, which in turn yields the factorization stated in the lemma.  
\end{proof}

We can now formulate a sufficient criterion for $\pi$ to do phase retrieval, based on suitable assumptions on the associated K-triple. The reasoning is rather similar to  Step 5 in the proof of Theorem \ref{thm:main_1}.

\begin{lemma} \label{lem:pr_inductive}
Let $G$ denote a $p$-group, and $\pi$ an irreducible, faithful representation of $G$, with a split K-triple $(G_0,r,\tau)$. Assume that $\tau$ does phase retrieval, and that $\pi$ is given by the realization from Lemma \ref{lem:ind_rep_explicit}.  
\begin{enumerate}
    \item[(a)] Assume that $p_0(\tau) < 1/2$. If $\eta_1 \in l^2(\mathbb{Z}_{p^r})$ has the phase retrieval property for the Schr\"odinger representation, and $\eta_2 \in \mathcal{H}_\tau$ has the phase retrieval property for $\tau$, and additionally $p_0(\tau,\eta_2) <1/2$, the vector 
    \[
    \eta \in l^{2}\left(  \mathbb{Z}_{p^r},\mathcal{H}_\tau \right)~,~\eta(j) = \eta_1(j) \cdot \eta_2
    \] does phase retrieval property for $\pi$. Such vectors exist. 
    \item[(b)] Pick any system $W \subset G_0$ of representatives modulo $\langle \{ y \} \cup Z(G) \rangle$, where $x,y \in G$ are associated to the K-triple according to Definition \ref{defn:K-triple}. Assume that
    \[
    \forall w \in W~:~ \alpha_1(w) = w
    \] holds, where $\alpha_1:W \to W$ is the bijection from Lemma \ref{lem:conj_action}. If $\eta_1 \in l^2(\mathbb{Z}_{p^r})$ has the phase retrieval property for the Schr\"odinger representation, and $\eta_2 \in \mathcal{H}_\tau$ has the phase retrieval property for $\tau$, the vector
        \[
    \eta \in l^{2}\left(  \mathbb{Z}_{p^r},\mathcal{H}_\tau \right)~,~\eta(j) = \eta_1(j) \cdot \eta_2
    \] has the phase retrieval property for $\pi$. 
\end{enumerate}
\end{lemma}

\begin{proof}
We first derive formulae for matrix coefficients, which are valid for both cases under consideration. Given $\eta = \eta_1 \cdot \eta_2$ with $\eta_1 \in l^2(\mathbb{Z}_{p^r})$ and $\eta_2 \in \mathcal{H}_\tau$, the matrix coefficient of $f \in l^2(\mathbb{Z}_{p^r}, \mathcal{H}_\tau)$ is computed in the following two-step procedure: Using Lemma \ref{lem:ind_rep_explicit} for all $k,l =0,\ldots,p^{r}-1$, with suitable choices of $x,y,z \in G$ and $W \subset G_0$ as in the cited lemma, we have for all $w \in W$ and all $k,l =0,\ldots,p^{r}-1,$
\begin{equation} \label{eqn:matrix_coeff_finite}
V_\eta f(y^kwx^l) = \sum_{j=0}^{p^r-1} \langle f(j) , \tau(x^{-j} w x^j) \eta_2 \rangle \xi_1^{-k} \xi_2^{kj} \overline{\eta_1(j-l)}~.  
\end{equation}

This can be rephrased as the result of consecutive computations of matrix coefficients with respect to $\tau$ and a (slightly modified version of) the Schr\"odinger representation. For this purpose, we let 
\begin{equation} \label{eqn:def_F_finite} 
F : W \times \mathbb{Z}_{p^r} \to \mathbb{C}~, F(w,j) = \langle f(j) , \tau(x^{-j} w x^j) \eta_2 \rangle
\end{equation} and note that 
\begin{equation} \label{eqn:F_mc_finite}
F(w,j) = V_{\eta_2}^\tau (f(j)) (x^{-j}wx^j)~.
\end{equation}
Then equation (\ref{eqn:matrix_coeff_finite}) can be rewritten as 
\begin{equation} \label{eqn:V_eta_rho_finite}
V_\eta f(y^kwx^l) = V_{\eta_1}^\varrho (F(w,\cdot)) (k,l)
\end{equation}
where 
\[
V_{\eta_1}^\varrho g (k,l)  =  \sum_{j=0}^{p^r-1} g(j) \xi_1^{-k} \xi_2^{kj} \overline{\eta_1(j-l)}~.
\] This is just a matrix coefficient with respect to the Schr\"odinger representation from Example \ref{ex:finite_Heisenberg}, modified by multiplication with $\xi_1^{-k}$, and by replacing the powers of the standard primitive $p^r$th root of unity, given by  $e^{2 \pi i/p^r}$, by powers of $\xi_2$, which is a primitive root of order $p^r$ as well, by Lemma \ref{lem:ind_rep_explicit}(b). Hence, $V_{\eta_1}^\varrho g$ coincides with the windowed Fourier transform of $g$ with respect to $\eta_1$ up to phase factors (independent of $g$) and a permutation of coefficients. By Lemma \ref{lem:basics_pr} (c), these operations do not affect the phase retrieval property, hence our choice of $\eta_1$ guarantees the phase retrieval property for $V_{\eta_1}^\varrho$ as well. 

Now assume that $f,h \in l^2(\mathbb{Z}_p^r,\mathcal{H}_\tau)$ are such that $|V_\eta f| \equiv |V_\eta g|$ holds. We use $F$ from (\ref{eqn:def_F_finite}) and define $H$  analogously as 
\begin{equation}
H: W  \times \mathbb{Z}_{p^r} \to \mathbb{C}~, H(w,j) = \langle h(j) , \tau(x^{-j} w x^j) \eta_2 \rangle
\end{equation} 
equation (\ref{eqn:V_eta_rho_finite}) gives for all $w \in W$
\[
 \left| V_{\eta_1}^\varrho (F(w,\cdot)) \right| =  \left| V_{\eta_1}^\varrho (H(w,\cdot)) \right| ~.
\] Since $\eta_1$ does phase retrieval, we obtain the existence of a map 
\[
u: W \to \mathbb{T}~,~ H(w,\cdot) = u(w)F(w,\cdot)~,~ \forall w \in W~. 
\] This implies for all $j \in \mathbb{Z}_{p^r}$ that
\[
|H(\cdot,j)| \equiv |F(\cdot,j)|~.
\] $W \subset G_0$ is a system of representatives modulo the central subgroup $\langle \{ y \} \cup Z(G) \rangle$ of $G_0$, and the latter is contained in the projective kernel of $\tau$. Hence equation (\ref{eqn:F_mc_finite}) (and its analog for $H$) allows to rewrite the previous equation as follows, for all $j \in \mathbb{Z}_{p^r}$:  
\[
\forall g_0 \in G_0 ~:~ \left| V_{\eta_2}^\tau (f(j)) (x^{-j} g_0 x^j) \right| = \left| V_{\eta_2}^\tau (h(j)) (x^{-j} g_0 x^j) \right|~.
\] For fixed $j \in \mathbb{Z}_{p^r}$, conjugation by $x^j$ is an automorphism of $G_0$, hence we may simplify this equation to 
\[
\left| V_{\eta_2}^\tau (f(j)) \right| = \left| V_{\eta_2}^\tau (h(j)) \right|~.
\] Now the phase retrieval property of $\eta_2$ implies the existence of a map 
\begin{equation} \label{eqn:role_v}
v : \mathbb{Z}_{p^r} \to \mathbb{T}~,\forall j \in \mathbb{Z}_{p^r}:~h(j) = v(j) f(j)~.
\end{equation} Note that this also entails, for all $j$, that 
\[
H(\cdot,j) \equiv v(j) F(\cdot,j)~.
\]
Defining the set
\begin{equation} \label{eqn:def_Cf}
    C_f = \{ j \in \mathbb{Z}_{p^r}~:~ f(j) \not= 0 \} ~,
\end{equation} (\ref{eqn:role_v}) shows that the proof of phase retrieval boils down to showing that $v$ is constant on $C_f$. This is achieved if we can show that for all 
$j,j' \in C_f$ there exists $w \in W$ such that 
\[
F(w,j) \cdot F(w,j') \not= 0 ~,
\] since then the reasoning finishing step 5 in the proof of Theorem \ref{thm:main_1} can be easily adapted to the finite setting: We obtain that 
\[
v(j) F(w,j) = H(w,j) = u(w) F(w,j)~,~ v(j') F(w,j') = u(w) F(w,j')
\] and $F(w,j), F(w,j')$ can be cancelled to give $v(j) = u(w) = v(j')$.

It is only at this stage that we need to distinguish the parts (a) and (b) of the lemma.
In the case of part (a), we recall that $F(\cdot,j)$ is the restriction to $W$ of a matrix coefficient with respect to the action of the representation $\tau$ on $\eta_1$. By choice of $W$ and $\eta_1$, and in view of Remark \ref{rem:p_0_proj_kernel}, the proportions of zero entries in $F(\cdot,j)$ and $F(\cdot,j')$ is therefore $<1/2$, for $j,j' \in C_f$. Hence the pidgeonhole principle yields the desired $w \in W$ with $F(w,j)F(w,j') \not=0$.

In case (b) we have $\alpha_j(w) = w$ for all $j \in \mathbb{Z}_{p^r}$. Here all $F(\cdot,j)$ are restrictions of matrix coefficients to $W$, with respect to the action of $\tau$ on $\eta_2$. Now the phase retrieval property of $V_{\eta_2}^\tau$ entails via Lemma \ref{lem:basics_pr} (b) that $F(\cdot,j) \cdot F(\cdot,j') \not\equiv 0$ holds for $j,j' \in C_f$.

We close the proof by observing that in case (a), the existence of a vector $\eta_2$ fulfilling the assumptions of (a) is guaranteed by Lemma \ref{lem:Zariski}.
\end{proof}

% 
% \begin{remark}
% As a first consequence of the lemma, we obtain a new proof of Theorem \ref{thm:class_2}, thus showing that our approach provides a proper generalization of the known results.
% 
% We first prove the Theorem for $p$-groups of nilpotency class 2. Here we proceed by induction over $|G|$. If $\pi$ is not faithful, we can consider it as a representation of the the strictly smaller quotient group $G/{\rm ker}(\pi)$. Since this group is also of nilpotency class 2, the induction hypothesis applies to yield phase retrieval for $\pi$.
% 
% Assuming that $\pi$ is faithful, we let $(G_0,r,\tau)$ denote a K-triple associated to $\pi$. $G_0$ is again nilpotent of class 2, hence the induction hypothesis applies to $G_0$, and yields phase retrieval for $\tau$. 
% 
% Since $G$ is of nilpotency class 2, we have $G/Z(G)$ is abelian. As a consequence, the conjugation of $G$ on $G_0/\langle \{ y \} \cup  Z(G) \rangle$ is trivial. Hence the conditions of Lemma \ref{lem:pr_inductive}(b) are fulfilled, allowing to conclude the phase retrieval property for $\pi$.
% 
% Thus the $p$-group case is covered. The general case now follows from this and Lemma \ref{PR copy(7)}. 
% \end{remark}

The next definition extends the bookkeeping device provided by K-triples, in order to account for repeated applications of Lemma \ref{lem:pr_inductive}.

\begin{definition}
Let $G$ be a $p$-group, and $\pi$ an irreducible representation. A \textbf{full sequence of K-triples associated to $\mathbf{\pi}$} is a sequence of triples $(H_i,r_i,\tau_i)$ with the following properties: Let $\tau_0 = \pi$, $K_0 = {\rm ker}(\pi_0), H_0 = G/K_0$, as well as $K_i = {\rm ker}(\tau_i)$ for $i=1,\ldots,k$. Then, for all $0 \le i < k$, $ (H_{i+1},r_{i+1},\tau_{i+1})$ is a K-triple for the representation $\overline{\tau}_i$ of $H_i/K_i$ induced by $\tau_i$, and $H_k/K_k$ is abelian. 
\end{definition}

\begin{remark} \label{rem:fs_Kt}
Repeated application of Lemma \ref{lem:exists_K_triple} shows that every irreducible representation $\pi$ has an associated full sequence of K-triples. 

The sizes of the quotient groups are related by the estimate
\[
|H_{i+1}/K_{i+1}| \le |H_{i+1}| = p^{-r_{i+1}} |H_i/K_i|~,
\] which results from the fact that $H_{i+1}< H_i/K_i$ has index $p^{r_{i+1}}$.
We further note that $H_k$ must be nontrivial: Recall that $\overline{\tau}_{k-1} \simeq {\rm ind}_{H_k}^{H_{k-1}/K_{k-1}} \tau_k$ is required to be irreducible. In the case of a trivial subgroup $H_k$, this representation is just the left regular representation of the (nontrivial) group $H_{k-1}/K_{k-1}$, which is never irreducible.

Hence $H_k$ has cardinality at least $p$. 
As a consequence we get that the length $k$ of a full sequence of K-triples fulfills 
\begin{equation} \label{eqn:est_length}
1+k \le \log_p|G|~.
\end{equation}
\end{remark}

Repeated applications of Lemma \ref{lem:pr_inductive} requires control over properties of the associated matrix coefficients, as quantified in $p_0(\pi)$. Since the quantity $p_0(\pi)$ is of independent interest
(see Remark \ref{rem:p_0}), the following lemma is as well. 
\begin{lemma} \label{lem:est_p0}
Let $G$ denote a nonabelian $p$-group, and $\pi$ a faithful, irreducible representation of $G$. Let $(G_0,r,\tau)$ denote a split K-triple associated to $\pi$. Then we have the estimate
\[
p_0(\pi) \le p_0(\tau) + \frac{p^r-1}{p^{2r}}~.
\]
\end{lemma}
\begin{proof}
We use the realization of $\pi$ from Lemma \ref{lem:ind_rep_explicit}. Fix a system $W \subset G_0$ of representatives modulo $\langle \{ y \} \cup Z(G) \rangle$. Given a nonzero $f \in l^2(\mathbb{Z}_{p^r},\mathcal{H}_\tau)$, we intend to estimate the proportions of zeros in the family $(V_\eta f (y^kwx^l))_{k,l=0,\ldots,p^r-1,w \in W}$. Since this is the restriction of $V_\eta f$ to a system of representatives modulo the center of $G$, the proportion of zeros of $V_\eta f$ coincides with this quantity.

We pick $\eta_1\in l^2(\mathbb{Z}_{p^r})$ and $\eta_2 \in \mathcal{H}_\tau$, and let $\eta = \eta_1 \cdot \eta_2 \in l^2(\mathbb{Z}_{p^r},\mathcal{H}_\tau)$. We recall some notations and observations from the proof of Lemma \ref{lem:pr_inductive}, specifically
\begin{equation} \label{eqn:def_F_finite_2} 
F : W \times \mathbb{Z}_{p^r} \to \mathbb{C}~, F(w,j) = \langle f(j) , \tau(x^{-j} w x^j) \eta_2 \rangle
\end{equation} and the related formulas 
\begin{equation} \label{eqn:F_mc_finite_2}
F(w,j) = V_{\eta_2}^\tau (f(j)) (x^{-j}wx^j)~,
\end{equation}
and 
\begin{equation} \label{eqn:V_eta_rho_finite_2}
V_\eta f(y^kwx^l) = V_{\eta_1}^\varrho (F(w,\cdot)) (k,l)~.
\end{equation}

We assume that $\eta_1,\eta_2$ are chosen such that 
\[
p_0(\sigma,\eta_1) = p_0(\sigma) = \frac{p^r-1}{p^{2r}} ~,~
p_0(\tau,\eta_2) = p_0(\tau)~,~
\] where $\sigma$ denotes the Schr\"odinger representation from Example \ref{ex:finite_Heisenberg}, and the second equality for $p_0(\sigma)$ follows by equation (\ref{eqn:p_0_Schroedinger}). Note that this quantity is also an upper bound for the proportion of zeros in $V_{\eta_1}^\varrho (g)$ for nonzero vector $g \in l^2(\mathbb{Z}_{p^r})$, by the same reasoning as in the proof of Lemma \ref{lem:pr_inductive}: The matrix coefficients $V_{\eta_1}^\varrho g$ and $V_{\eta_1}^\sigma g$ differ up to a multiplication with phase factors and a permutation, none of which affects the proportion of zeros. 

Let
\[
c_0 = \frac{ |\{ w \in W: F(w,\cdot) \equiv 0 \}|}{|W|}~.
\]
In view of (\ref{eqn:F_mc_finite_2}) and (\ref{eqn:V_eta_rho_finite_2}) we have that
\[
V_\eta^\pi f (y^kwx^l) = 0 
\] may occur for fixed $w \in W$ and suitable $k,l$ either if $F(w,\cdot) \equiv 0$, in which case it holds for all $k,l \in \mathbb{Z}_{p^r}$, or if $F(w,\cdot) \not\equiv 0$ holds. In the latter case, the proportion of zeros in $V_\eta^\pi  f(y^kwx^l)$ for fixed $w$ can be bounded using $p_0(\eta_2,\sigma)$. These observations lead to the overall estimate
\begin{eqnarray*}
\lefteqn{|\{ (k,w,l) : V_\eta^\pi f(y^kwx^l) = 0 \}|} \\ & \le & |\{ w \in W : F(w,\cdot) = 0 \}| \cdot p^{2r} + |\{ w \in W : F(w,\cdot) \not= 0 \}| \cdot p^{2r} p_0(\eta_2,\sigma) ~.
\end{eqnarray*}
Dividing both sides of this inequality by $|G/Z(G)| = |W|p^{2r}$ yields
\[
\frac{|\{ (k,w,l) : V_\eta^\pi f(y^kwx^l) = 0 \}|}{|W|p^{2r}} \le c_0 + (1-c_0)  \frac{p^r-1}{p^{2r}} \le c_0 + \frac{p^r-1}{p^{2r}}~.
\]

It therefore remains to prove
\begin{equation} \label{eqn:est_c0}
c_0 \le p_0(\eta_2,\tau) = p_0(\eta_2)~.
\end{equation}
To see this, pick any $j_0 \in \mathbb{Z}_{p^r}$ such that $f(j_0) \not= 0$; $j_0$ exists since we assume $f$ to be nonzero. 
We can then estimate 
\begin{eqnarray*}
|\{ w \in W : F(w,\cdot) \equiv 0 \} | & \le & |\{ w \in W: F(w,j_0) = 0 \}| \\
& = &  |\{ w \in W : V_{\eta_2}^\tau (f(j_0)) (x^{-j_0}wx^{j_0})  = 0 \}| \\
& \le & |W|\cdot p_0(\tau,\eta_2)~,
\end{eqnarray*}
which proves  (\ref{eqn:est_c0}), and finishes the proof of the lemma.
\end{proof}

A straightforward inductive application of Lemma \ref{lem:pr_inductive} now gives the following result, which formulates a sufficient criterion for phase retrieval that can be checked using a full sequence of K-triples.  
\begin{corollary} \label{cor:p_groups_general}
Let $G$ denote a $p$-group. Assume that $\pi$ is an irreducible representation of $G$ possessing a full sequence of K-triples $(H_i,r_i,\tau_i)$, for $i=1,\ldots,k$, that are all split. If 
\[
\sum_{i=2}^k \frac{p^{r_i}-1}{p^{2r_i}} < \frac{1}{2}~,
\] then $\pi$ does phase retrieval. 
\end{corollary}
\begin{proof}
A downwards induction establishes
\[
p_0(\tau_j) \le \sum_{i=j+1}^k \frac{p^{r_i}-1}{p^{2r_i}} < \frac{1}{2}~, 
\] for all $j=1,\ldots,k$. Hence Lemma \ref{lem:pr_inductive} (a) can be applied $k$ times, and yields the desired statement. 
\end{proof}

As an application of the corollary, we show that a simple size estimate on $|G|$ suffices to guarantee phase retrieval for all irreducible representations of $G$, if $G$ has exponent $p$.

\begin{corollary} \label{cor:small_p_groups}
Assume that the $p$-group $G$ has exponent $p$ and fulfills 
$|G| \le p^{2+p/2}$. Then every irreducible representation of $G$ does phase retrieval. 
\end{corollary}
\begin{proof}
First note that, $p$ being a prime, the exponent $p$ property is equivalent to stating that every nontrivial group element has order $p$.
Clearly this property is inherited by quotients and subgroups. As a consequence, every K-triple is split and of type $(G_0,1,\tau)$, and this observation extends to all K-triples occurring in a full sequence associated to an irreducible representation $\pi$. 

Let $k$ denote the length of such a full sequence of K-triples. By equation (\ref{eqn:est_length}) and the assumption on $|G|$ we have 
\[ k \le \log_p |G| - 1 \le 1+p/2 ~.\] Using that all exponents in the sequence of $K$-triples equal one, the computation 
\begin{eqnarray*}
\sum_{i=2}^k \frac{p-1}{p^{2}} & = &  (k-1) \frac{p-1}{p^2} \\
& \le & (\log_p|G|-2) \cdot \frac{p-1}{p^2} \le \frac{p}{2} \cdot \frac{p-1}{p^2} < \frac{1}{2}~.
\end{eqnarray*}
shows that Corollary \ref{cor:p_groups_general} applies, and yields the result. 
\end{proof}

\begin{remark}
%The group $G$ and its representation $\pi$ from Example \ref{ex:finite_Kirillov} are an illustration that the phase retrieval property of $\pi$ does not entail that the associated K-triple is split.
It is quite possible that the requirement concerning split K-triples reflects a shortcoming of our proof method, rather than being an objective obstacle for the phase retrieval property. Understanding the precise relationship between these properties remains to be investigated in future work. 

Given that it is currently unclear which irreducible representations $\pi$ of a general $p$-group have an associated sequences of K-triples that are also split, the true scope of Corollary  \ref{cor:p_groups_general} is not yet fully understood. However, Corollary \ref{cor:small_p_groups} exhibits one general class to which Corollary \ref{cor:p_groups_general} applies, and it is already quite sizeable: According to \cite{MR3416635}, the number of $p$-groups of order $p^8$ and exponent $p$, for primes $p >7$, is given by the formula
\[
 p^4+2p^3+20p^2+147p+ (3p+29) \gcd(p-1,3) + 5 \gcd(p-1,4)+1246~,
\] and Corollary \ref{cor:small_p_groups} guarantees the phase retrieval properties for arbitrary irreducible representations of these groups, as soon as $p\ge 13$. The enumeration and systematic construction of $p$-groups with specific properties is still an ongoing endeavour, and therefore the emergence of further general classes which can be treated with Corollary \ref{cor:p_groups_general} seems quite plausible. 

In any case, we expect that Corollary \ref{cor:p_groups_general} can be employed for the systematic construction of $p$-groups and associated representations with particular interesting properties. One such property is the nilpotency class of $G$: Observe that the size restriction $|G| \le p^{2+p/2}$ entails a restriction on the nilpotency class. However, we conjecture that for any prime number $p$ there exist $p$-groups $G$ of arbitrarily large nilpotency class and/or size, and irreducible representations of $G$ doing phase retrieval. A proof of this conjecture via Corollary \ref{cor:p_groups_general} would be achieved by constructing sequences of split K-triples of arbitrary length $k$, with sufficiently large exponents $r_1,\ldots, r_k$. 
\end{remark}

\section*{Concluding remarks}

While the statements of our main results were \textit{existence results} for vectors doing phase retrieval, the inductive methods employed to prove these results provide access to the explicit construction of such vectors. 

In the Lie group case, this can be formulated in a rather satisfactory way. If one follows the inductive construction of an arbitrary irreducible representation by repeated applications of Kirillov's lemma and induction of representations, the canonical identification of Hilbert spaces $L^2(\mathbb{R},L^2(\mathbb{R}^k)) \cong L^2(\mathbb{R}^{k+1})$ ultimately provides a realization of $\pi$ on $\mathbb{R}^m$ for a suitable $m\ge 1$. If one traces the inductive construction of the vector $\eta$ doing phase retrieval in Step 3, one can employ a Gaussian for $\eta_1$ at each step. In the identification $L^2(\mathbb{R},L^2(\mathbb{R}^k)) \cong L^2(\mathbb{R}^{k+1})$ one obtains $\eta (t,u) = \eta_1(t) \eta_2(u)$, and a simple induction yields that $\eta$ is a Gaussian, as well. 

This leads to the fundamental observation that every irreducible representation of an arbitrary simply connected, connected nilpotent Lie group has an explicitly computable realization on $L^2(\mathbb{R}^m)$, for suitable $m$, in which Gaussian vectors do phase retrieval. 

In the finite group case, the setting is somewhat more involved, even assuming that all elements in a full sequences of K-triples are already known to be split: Note that the corresponding induction step in Lemma \ref{lem:pr_inductive} requires choosing a vector $\eta_2$ that guarantees phase retrieval \textbf{and} $p_0(\eta_2, \tau) < 1/2$, and combining it with a vector $\eta_1$ doing phase retrieval for the Schr\"odinger representation. Here the proof of Lemma \ref{lem:est_p0} estimating $p_0(\eta_2,\tau)$ relies on the same procedure to construct $\eta$ from $\eta_1$ and $\eta_2$. 
That means that the vectors $\eta$ simultaneously guaranteeing the estimates in Lemma \ref{lem:est_p0} and phase retrieval can be constructed inductively as well, as soon as one can solve the following, currently open problem: Are there explicit constructions of vectors $\eta$ for the Schr\"odinger representation which simultaneously have the phase retrieval and full spark property?

\bibliographystyle{abbrv}
\bibliography{pr_nilpotent}

\end{document}